\documentclass{article}

\usepackage[letterpaper,top=2cm,bottom=2cm,left=3cm,right=3cm,marginparwidth=1.75cm]{geometry}

\usepackage{amsmath}
\usepackage{amsthm}
\usepackage{amssymb}
\usepackage{amsfonts}
\usepackage{authblk}
\usepackage{bm}
\usepackage{graphicx}
\usepackage[colorlinks=true, allcolors=blue]{hyperref}
\usepackage{mathtools}
\usepackage[english]{babel}
\usepackage{needspace}
\usepackage{upgreek}

\newtheorem{theorem}{Theorem}[section]
\newtheorem{definition}[theorem]{Definition}
\newtheorem{corollary}[theorem]{Corollary}
\newtheorem{proposition}[theorem]{Proposition}
\newtheorem{lemma}[theorem]{Lemma}

\AtBeginEnvironment{theorem}{\Needspace{8\baselineskip}}
\AtBeginEnvironment{definition}{\Needspace{8\baselineskip}}
\AtBeginEnvironment{corollary}{\Needspace{8\baselineskip}}
\AtBeginEnvironment{proposition}{\Needspace{8\baselineskip}}
\AtBeginEnvironment{lemma}{\Needspace{8\baselineskip}}
\newcommand{\reals}{\mathbb{R}}

\newcommand{\mb}{\mathbf}
\DeclarePairedDelimiter{\norm}{\lVert}{\rVert} 
\DeclareMathOperator{\Col}{col}
\DeclareMathOperator{\rank}{rank}
\DeclareMathOperator{\Span}{span}

\title{Linear Systems and Eigenvectors in Constructive Mathematics}
\author[ ]{Bob Driessen \hspace{-0.25cm}}
\author[1]{Christiaan J. F. van de Ven}
\affil[1]{Radboud University, Institute for Mathematics, Astrophysics and Particle Physics,
Heyendaalseweg 135, 6525 AJ Nijmegen (The Netherlands)
}
\date{\today}

\begin{document}

\maketitle

\begin{abstract}
\noindent
In this work we study two classical problems of (numerical) linear algebra: (i) solving linear systems and (ii) computing eigenvectors, within a constructive framework.
Numerical accuracy and indeterminacy are naturally incorporated through Bishop-style constructive mathematics. 
Our contributions include new results on Gauss-Jordan elimination and on approximating the rank of a matrix. 
Additionally, we introduce a novel method for constructing approximate eigenvectors, based on a previously unexplored characterization of singular matrices.
\end{abstract}

\section{Introduction}

Many publications on constructive linear algebra exist.
The main references are \cite{CiM2} and \cite{ACiCA}, both of which are mainly based on Wim Ruitenburg's Ph.D. thesis \cite{InA}.
An older but almost forgotten article is \cite{UuIA} by Heyting. 
The article \cite{DoCMaaRoI} by Scedrov addresses a special case of matrix diagonalization. 
Its constructive results are used primarily as tools for studying classical function spaces through topological models. 
Nevertheless, the constructive approach to linear algebra presented in Scedrov's work is original.
Finally, \cite{CAoEPiCuNU} considers the calculation of eigenvalues and eigenvectors.

Another related line of research is linear algebra in computable analysis.
The book chapter \cite{FoITitEoNaMC} deals with solving linear equations, and \cite{CtDoLS} is about computing the dimension of a linear space.
While here the metalanguage is classical, the results have a flavor similar to constructive mathematics.

However, none of these works address Gauss-Jordan elimination. The present article fills this gap.
The main results of this work can be found in Section \ref{sec_gauss} and Section \ref{sec_eigen}. 
To reach these, we start with a basic introduction to constructive real numbers in Section \ref{sec_basic_def}. 
This also includes a proof of the equivalence of co-transitivity of equality with WLLPO.
Section \ref{sec_vector} subsequently introduces the theory of finite vector spaces.
Although this topic is also covered in \cite{CiM2}, there are some omissions in Troelstra and Van Dalen's treatment, e.g.
they do not mention the independent version of the Austauschsatz (Lemma \ref{l_independent_austausch}).
However, we need it to prove that a ranked matrix has a basis consisting of a subset of its own columns, as in \cite{UuIA}.
In Section \ref{sec_gauss} we consider Gauss-Jordan elimination.
The direct translation of this procedure into a constructive setting is equivalent to an omniscience principle (Proposition \ref{p_partial_pivoting_lpo}).
Gaussian elimination with partial pivoting, i.e. using only row operations, is not constructively valid.
However, as common in this field, a suitably adjusted version still holds in the form of Theorem \ref{t_gauss}.
To achieve this, we also allow column swaps, corresponding to full pivoting in numerical linear algebra. For general matrices, where $\rank A$ may not be computable, the best result we can achieve is an $\varepsilon$\textit{-reduced row echelon form} (Proposition \ref{p_epsilon_reduction}).

A key contribution of Section \ref{sec_eigen} is a novel construction of approximate eigenvalues.
From $\varepsilon$-reduced row echelon forms we obtain a new characterization of singular matrices, see Theorem \ref{t_det_zero}.
This characterization already implies the existence of $\varepsilon$-eigenvectors for every $\varepsilon > 0$.
However, establishing the linear independence of these $\varepsilon$-eigenvectors for sufficiently small $\varepsilon$ in Theorem \ref{t_eigen} requires additional work.

To help readers unfamiliar with \textit{constructive} mathematics, we have kept this article largely self-contained.
Proofs of known and/or elementary results are included.
However, this work is obviously not an introduction to constructive mathematics.
For the necessary context and more extensive explanations, the reader should consult \cite{ToCA}.
\section{Preliminaries} \label{sec_basic_def}
In this section we recall some basic facts.
Normal font variables like $a, b$ denote scalars in $\reals$.
Equality, addition, subtraction, multiplication and division of numbers in $\reals$ are defined and denoted just as in classical mathematics.
As is customary in constructive mathematics, $x \neq y$ does \textbf{NOT} mean $\neg (x = y)$.
Instead, it means that $x$ and $y$ are \textit{apart}, there is a rational $q > 0$ such that $| x - y | > q$.
This notion provides a positive formulation of inequality, which behaves similarly to classical inequality in most respects.
Although we generally cannot decide whether $x = 0$ or $x \neq 0$, for every $\varepsilon > 0$ we can assert that either $|x| < \varepsilon$, or $x \neq 0$. 
This weaker form of decidability is often sufficient for proofs where, classically, one would rely on the law of excluded middle $x = 0 \lor x \neq 0$. 
However, the classical formulation $\neg(x = y)$ lacks an important property called co-transitivity (see Lemma \ref{l_apartness_cotransitive}). 
We include a proof of this lemma here, as it does not appear with a full proof in the literature. 
The only reference available is \cite{CRM}, p. 45.
Although the result is fairly straightforward, its inclusion is meaningful. 
If the inequality of real numbers were co-transitive, then vectors in $\mathbb{R}^n$ would naturally inherit this property. 
In that case, one could redefine the notion of independence of vectors (Definition \ref{d_indep}) using inequality, which would fundamentally alter the character of the constructive theory of finite-dimensional vector spaces. 
Obviously this is not the case, which is an intuitive and expected result, but it still requires proof.

The other sections of this article do not rely on a specific construction of the real numbers.
Examples are Cauchy sequences of rational numbers or the construction in section 2.1 of \cite{ToCA}.
We simply assume that a suitable structure $\reals$ with the desired algebraic, topological, and ordering properties exists.

However, to prove Lemma \ref{l_apartness_cotransitive}, we need a concrete realization of $\reals$. 
For this purpose, we choose sequences of rational, closed intervals that are nested and dwindling. 
Rational numbers remain simple and well-behaved objects in constructive mathematics: their equality and ordering are \textit{decidable}, meaning that for any pair of rational numbers $a, b$, we have $a = b \lor a \neq b$. 
In other words, for rational numbers, $\neg(a = b)$ and $a \neq b$ are equivalent. 
We will use these properties extensively in the following definitions and proofs. 
The following definitions will be used throughout this paper.

\begin{definition}
Let $[a, b]$ and $[c ,d]$ be rational, closed intervals.
We say $[a, b]$ and $[c ,d]$ \textbf{coincide} if $[a ,b] \cap [c ,d]$ contains a rational number.
The \textbf{length} of a rational, closed interval $[a ,b]$ is $b - a$ and we denote this with $l([a, b])$.
\end{definition}

\begin{definition}
A sequence of rational, closed intervals $(x_n)_n$ is \textbf{nested} if $x_{n+1} \subseteq x_n$ for all $n$.
It is \textbf{dwindling} if for all $n$ there is an $m$ such that $l(x_{k}) < 1/2^n$ whenever $k \geq m$.
\end{definition}

\begin{definition} \label{d_real}
A \textbf{real number} $x$ is a sequence of nested and dwindling rational, closed intervals $(x_n)_n$. We define its endpoints by
\[
x_n^l := \min x_n, \qquad x_n^r := \max x_n.
\]
Two real numbers $x$ and $y$ are \textbf{equal} if $x_n$ and $y_n$ coincide for all $n$ and we write $x = y$.
We write $x < y$ if there is an $n$ such that $x_n^r < y_n^l$, and $x > y$ is defined as $y < x$.
If $x < y$ or $x > y$ then we say $x$ and $y$ are apart and we denote this with $x \neq y$.
The statement $x \leq y$ means $\neg (x > y)$ and likewise $x \geq y$ is defined as $y \leq x$.
\end{definition}

\begin{definition}[Addition]
Let $x$ and $y$ be two real numbers. 
Their \textbf{sum } $\mathbf{x} + \mathbf{y}$ is the real number defined by $(x + y)_n^l = x_n^l + y_n^l$ and $(x + y)_n^r = x_n^r + y_n^r$.
\end{definition}

\begin{definition}[Multiplication]
Let $x$ and $y$ be two real numbers. 
Their \textbf{product} $\mathbf{x}  \mathbf{y}$ is the real number defined by $(x  y)_n^l = \min(x_n^l  y_n^l, x_n^l  y_n^r, x_n^r  y_n^l, x_n^r  y_n^r)$ and $(x  y)_n^r = \max(x_n^l  y_n^l, x_n^l  y_n^r, x_n^r  y_n^l, x_n^r  y_n^r)$.
\end{definition}

\begin{definition}[Absolute value]
Let $x$ be a real number, then the real number $\mathbf{\lvert x \rvert}$ is defined by $\lvert x \rvert_n = x_n$ if $0 \in x_n$ and $\lvert x \rvert_n^l = \min(\lvert x_n^l \rvert, \lvert x_n^r \rvert )$ and $\lvert x \rvert_n^r = \max(\lvert x_n^l \rvert, \lvert x_n^r \rvert)$ otherwise.
\end{definition}

\begin{lemma} \label{l_eq_zero}
Let $x \in \mathbb{R}$ and suppose $\lvert x \rvert < q$ for all $q \in \mathbb{Q}^+$, then $x = 0$.
\end{lemma}
\begin{proof}
We have to show that the rational interval $x_n$ contains $0$ for all $n$.
So, let $m$ be an arbitrary index, then $x_m = [x_m^l, x_m^r]$ be a rational interval.
That means $0 \in x_m$ is decidable.
Thus, we can prove $0 \in x_m$  by deriving a contradiction from $\neg (0 \in x_m)$.
In this case, either $0 < x_m^l$ or $0 > x_m^r$ holds.
By symmetry, we only consider $0 < x_m^l$.
We then easily see $\lvert x \rvert > \frac{1}{2} x_m^l$.
This contradicts $\lvert x \rvert < q$ for all $q \in \mathbb{Q}^+$ and hence $0 \in x_m$ . 
\end{proof}

\begin{lemma}[Apartness is co-transitive] \label{l_apartness_cotransitive}
Let $x, y \in \mathbb{R}$. 
If $x \neq y$, then for every $z \in \reals$ we have $x \neq z$ or $y \neq z$.
\end{lemma}
\begin{proof}
Without loss of generality, assume $x > y$.
This means there is an $n$ such that $x_n^l > y_n^r$.
Denote $q = x_n^l - y_n^r$, then $q > 0$ and we can find an $m$ such that $l(z_m) < q$.
We readily see that $z_m^r < x_n^l$ or $z_m^l > y_n^r$, as $l(z_m) < q$ implies that the interval $z_m$ can not overlap with both $x_m$ and $y_m$.
\end{proof}

Often in constructive mathematics we want to show that some result from classical mathematics is not acceptable constructively.
To prove this, we always apply the same recipe.
We show that the statement implies an unconstructive statement of a very simple kind.
Or something equivalent to such a simple statement.
These simple statements always involve the most basic yet infinite structures, sequences of zeros and ones.
An example is the statement `any sequence of zeroes and ones either contains a one, or every entry is zero'.
We can \textit{not} disprove this statement constructively.
It does seem unacceptable to include it in a constructive proof.
We can not constructively give either the entry of the first one, or prove that no such entry exists.
There can be no algorithm guaranteed to decide this problem in a finite time.
A being capable of deciding these problems has godlike powers.
This is why Bishop called these `omniscience' principles. 
Below we collect some of the more common of these principles.

\begin{definition}[Cantor space] \label{d_Cantor}
The set $\mathcal{C}$ or \textbf{Cantor space} consists of the infinite sequences of zeroes and ones.
The entries of an element $\alpha \in \mathcal{C}$ are denoted as $\alpha_n$, so $\alpha = (\alpha_n)_{n \in \mathbb{N}}$.
\end{definition}

\begin{definition}[LPO] \label{d_LPO}
The \textbf{Limited Principle of Omniscience} (LPO) is the statement
`Let $\alpha \in \mathcal{C}$, then either $\forall n: \alpha_n = 0$ or $\exists n: \alpha_n = 1$.'.
\end{definition}

\begin{definition} \label{d_LLPO}
The \textbf{Lesser Limited Principle of Omniscience} (LLPO) is the statement
`Let $\alpha \in \mathcal{C}$ contain at most one $1$, then either $\forall n: \alpha_{2n} = 0$ or $\forall n: \alpha_{2n + 1} = 0$.'.
\end{definition}

\begin{definition} \label{d_WLLPO}
The \textbf{Weaker Lesser Limited Principle of Omniscience} (WLLPO) is the statement
`Let $\alpha \in \mathcal{C}$ contain at most one $1$ and $\neg (\forall \alpha_n = 0)$, then either $\forall n: \alpha_{2n} = 0$ or $\forall n: \alpha_{2n + 1} = 0$.'.
\end{definition}

Having defined and explained the role of principles as WLLPO in constructive mathematics, we now state the main result of this section.
\clearpage
\begin{proposition}[Inequality can not be proven to be co-transitive]  \label{p_inequality_not_cotransitive}
Let $T$ be the statement `For all $x, y \in \mathbb{R}$, if $\neg (x = y)$, then for every $z \in \reals$ we have $\neg (x = z)$ or $\neg (y = z)$.'.
Then $T$ is equivalent to WLLPO. 
\end{proposition}
\begin{proof}
First, suppose $T$ holds.
Let $\alpha \in \mathcal{C}$ be such that $\alpha$ contains at most one $1$ and $\neg (\alpha = 0)$.
Define three real numbers $x$, $y$ and $z$ as follows:
\[
x_n = \begin{cases} 
[\frac{-1}{2^n}, \ \frac{1}{2^n}] & \text{if } \alpha_k = 0 \text{ for all } k < n \\
[\frac{1}{2^m}, \ \frac{1}{2^m}] & \text{if } \text{there is a } k < n \text{ such that } \alpha_k = 1 \text{ and } m \text{ is the least such } k \\
\end{cases}
\]
\[
y_n = \begin{cases} 
[\frac{-1}{2^n}, \ \frac{1}{2^n}] & \text{if } \alpha_k = 0 \text{ for all } k < n \\
[\frac{-1}{2^m}, \ \frac{-1}{2^m}] & \text{if } \text{there is a } k < n \text{ such that } \alpha_k = 1 \text{ and } m \text{ is the least such } k \\
\end{cases}
\]
\[
 z_n = \begin{cases} 
[\frac{-1}{2^n}, \ \frac{1}{2^n}] & \text{if } \alpha_k = 0 \text{ for all } k < n \\
[\frac{1}{2^m}, \ \frac{1}{2^m}] & \text{if } \text{there is a } k < n \text{ such that } \alpha_k = 1 \text{ and the least such } m \text{ is even } \\
[\frac{-1}{2^m}, \ \frac{-1}{2^m}] & \text{if } \text{there is a } k < n \text{ such that } \alpha_k = 1 \text{ and the least such } m \text{ is odd. } \\
\end{cases}
\]
Then $\neg (x = y)$ follows easily.
The assumption $x = y$ leads to $\forall n: \alpha_n = 0$, a contradiction.
Next, suppose that we can prove $\neg (x = z) \lor \neg (y = z)$.
By symmetry, we only consider $\neg (x = z)$.
Suppose $k$ satisfies $\alpha_k = 1$, then $k$ must be odd.
For if $k$ were even, then the construction of $x$ and $z$ implies $x = z$.
This contradicts $\neg (x = z)$, so we have shown $\alpha_n = 0$ for all even $n$.
\\\\
Next, assume WLLPO and let $x$ and $y$ be real numbers such that $\neg (x = y)$.
Given some other real number $z$, construct $\alpha \in \mathcal{C}$ as follows.
For all $n$ we only have $\alpha_{2n} = 1$ if $\alpha_m = 0$ for $m < n$ and $z_n$ and $x_n$ do not overlap.
Similarly $\alpha_{2n + 1} = 1$ if $\alpha_m = 0$ for $m < n$ and $y_n$ and $x_n$ do not overlap.
This construction ensures that $\alpha$ has at most one $1$.

We now prove that $\forall_n \alpha_n = 0$ is contradictory.
In this case, $z_n$ overlaps with both $x_n$ and $y_n$ for all $n$.
This implies $\lvert x - y \rvert < l(z_n)$ for all $n$ and Lemma \ref{l_eq_zero} then yields $x = y$.
Thus, we have a contradiction.
The sequence $\alpha \in \mathcal{C}$ therefore satisfies the conditions of WLLPO and we get $\alpha_n = 0$ for all even $n$ or $\alpha_n = 0$ for all odd $n$.
By symmetry, we only look at the first case.
The construction of $\alpha$ implies that $x_n$ and $z_n$ overlap for all $n$.
This means $x = z$ and therefore $\neg(z = y)$.
\end{proof}
\section{Vector spaces} \label{sec_vector}
We assume familiarity with basic notation and definitions of vector spaces in classical mathematics.
Boldface letters like $\mb{x}, \mb{y}, \mb{z}, \mb{a}_i$ and the Greek letters $\alpha, \beta, \gamma$ denote vectors in $\reals^n$. 
The symbol $\mb{n}$ denotes a vector whose entries are all equal to the number $n$, for example $\mb{0}=(0,0,\dots,0)$ denotes the zero vector. 
To refer to the entries of a vector, we write
\[
\mathbf{x} = 
\begin{vmatrix}
(\mb{x})_1 \\
(\mb{x})_2 \\
\vdots \\
(\mb{x})_n
\end{vmatrix}
\
\beta = 
\begin{vmatrix}
\beta_1 \\
\beta_2 \\
\vdots \\
\beta_n
\end{vmatrix}
\]
Vectors in $\reals^n$ inherit equality, addition, and subtraction from the real numbers as in classical mathematics.
Scalar multiplication also provides no difficulties.
Again, $\mb{x} \neq \mb{y}$ denotes apartness, so there is a $1 \leq i \leq n$ such that $x_i$ and $y_i$ are apart.
The \textit{span} of a set of vectors $X$ is denoted by $\Span X$.
If $Y$ is a set of vectors and $\Span X \subseteq Y$, then we say that $X$ \textit{spans} $Y$.
We start with some basic facts about vector addition and scalar multiplication.

\begin{lemma} \label{l_beta_entry_apart_zero}
Let $\mb{x}_1,...,\mb{x}_n$ be vectors in $\reals^m$.
Suppose $\beta \in \reals^m$ is such that $\sum_i \beta_i \mb{x}_i \neq \mb{0}$.
Then there is a $1 \leq j \leq n$ with $\beta_j \neq 0$.
\end{lemma}
\begin{proof}
Since $\sum_i \beta_i \mathbf{x}_i \neq \mathbf{0}$, there exists an index $k$ such that 
$(\sum_i \beta_i \mathbf{x}_i)_k \neq 0$. 
Hence, $\sum_i \beta_i (\mathbf{x}_i)_k \neq 0$. 
Since $x + y \neq 0$ implies $x \neq 0$ or $y \neq 0$ for real numbers, we can find a $j$ such that $\beta_j (\mb{x}_j)_k \neq 0$.
Similarly, $xy \neq 0$ implies $x \neq 0$ and $y \neq 0$. 
We conclude $\beta_j \neq 0$.
\end{proof}

\begin{lemma}[Apartness is co-transitive] \label{l_apartness_cotransitive_vec}
Let $\mb{x}, \mb{y} \in \reals^n$. 
If $\mb{x} \neq \mb{y}$, then for every $\mb{z} \in \reals^n$ we have $\mb{x} \neq \mb{z}$ or $\mb{y} \neq \mb{z}$.
\end{lemma}
\begin{proof}
This is a direct consequence of Lemma \ref{l_apartness_cotransitive} applied coordinate-wise.
\end{proof}

The classical notion of \textit{independence} needs a stronger formulation constructively.
This, of course, has everything to do with Proposition \ref{p_inequality_not_cotransitive}.
Co-transitivity is a fundamental attribute of apartness.
If we were to define independence using the weaker definition $\sum_i \beta_i \mb{x}_i = \mb{0}$ implies $\beta = \mb{0}$, then many normal properties of independence would not hold.
In \cite{CiM2} what we call \textit{independent} is referred to as \textit{free}.
With \textit{independent} Troelstra and van Dalen meant what we refer to as \textit{weak independence}.
Their nomenclature did not stick, as we usually want to denote the most natural analogue of a classical definition by its classical name.
For this reason we use the more standard constructive definition of independence as can be found in \cite{ToCA} for instance.

\begin{definition}\label{d_indep}
The vectors $\mb{x}_1, \mb{x}_2,...,\mb{x}_m$ are \textbf{independent} if for all $\beta \in \reals^m$, $\beta \neq \mathbf{0}$ implies $\sum_{i} \beta_i \mb{x}_i \neq 0$.
If $X$ is a set of independent vectors, then we will also call $X$ itself \textbf{independent}.
Finally, a vector $\mathbf{x} \in \mathbb{R}^n$ is said to be \textbf{independent from} $\mathbf{x}_1, \dots, \mathbf{x}_m$ if $\mathbf{x} \neq \sum_{i=1}^m \beta_i \mathbf{x}_i \quad \text{for all } \beta \in \mathbb{R}^m
$.
\end{definition}

\begin{lemma} \label{l_weakly_indep_equi_unique_repr}
If the set of vectors $X = \{ \mb{x}_1, \mb{x}_2,...,\mb{x}_m \}$ is independent, then every vector in $\Span X$ has a unique representation as a sum of vectors from $X$.
\end{lemma}
\begin{proof}
Let $X$ be independent and suppose that $\sum_i \beta_i \mb{x}_i = \sum_i \lambda_i \mb{x}_i$.
Then $\sum_i (\beta - \lambda)_i \mb{x}_i = \mb{0}$ and we shall prove $\beta_i - \lambda_i = \mb{0}$ for every $i$ by contradiction.
The assumption $\beta_j - \lambda_j \neq \mb{0}$ for some $j$ leads to $\sum_i (\beta - \lambda)_i \mb{x}_i \neq \mb{0}$ by the independence of $X$.
\end{proof}

\begin{lemma} \label{l_indep_add}
Let $\mb{x}_1,...,\mb{x}_n$ be a set of independent vectors and suppose $\mb{x}_0$ is independent from these.
Then the set $\mb{x}_0, \mb{x}_1,...,\mb{x}_n$ is also independent.
\end{lemma}
\begin{proof}
Let $\beta \neq \mb{0}$, we will prove $\sum_{i \geq 0} \beta_i \mb{x}_i \neq \mb{0}$. Since $\beta \neq \mb{0}$, there exists $j$ with $\beta_j \neq 0$. 
We distinguish two cases: either $j = 0$ or $j > 0$.
If $j = 0$, then  $\sum_{i \geq 0} \beta_i \mb{x}_i \neq \mb{0}$ follows from $\beta_0 \mb{x}_0 \neq - \sum_{i > 0} \beta_i \mb{x}_i$.
We obtain the latter apartness because $\mb{x}_0$ is independent from $\mb{x}_1,...,\mb{x}_n$.
In particular, this implies $\mb{x}_0 \neq \sum_{i > 0} -\beta_0^{-1} \beta_i \mb{x}_i$.
If $j > 0$, then we know $\sum_{i > 0} \beta_i \mb{x}_i \neq \mb{0}$ because $\mb{x}_1,...,\mb{x}_n$ are independent.
This means either $\beta_0 \mb{x}_0 \neq \mb{0}$ or $\beta_0 \mb{x}_0 \neq -\sum_{i > 0} \beta_i \mb{x}_i$ by Lemma \ref{l_apartness_cotransitive_vec}.
The second of these possibilities directly implies $\sum_i \beta_i \mb{x}_i \neq \mb{0}$.
If $\beta_0 \mb{x}_0 \neq \mb{0}$, then we know $\beta_0 \neq \mb{0}$.
Thus we are back at the case $j = 0$, which we already covered.
\end{proof}

As in classical mathematics, the Steinitz exchange lemma or Austauschsatz (Theorem \ref{t_steinitz}) plays a central role in finite dimensional vector spaces.
We need it to prove that any two bases of such vector spaces contain the same number of elements.
This is the content of Proposition \ref{p_basis_def}.
To get to this propositions, we first need to derive several Austauschsatz like results.

\begin{lemma}[Austauschlemma] \label{l_austausch}
Let $X = \{ \mb{x_1},...,\mb{x_n} \}$ be a set of vectors in $\reals^m$.
Suppose $\mb{y} \in \Span{X}$ and $\mb{y} \neq 0$.
Then there is an $1 \leq j \leq n$ such that swapping $\mb{x}_j$ for $\mb{y}$ in $X$ yields a set of vectors $Z$ whose span equals $\Span{X}$.
\end{lemma}
\begin{proof}
Let $\beta \in \reals^m$ be such that $ \sum_{i} \beta_i \mb{x}_i = \mb{y}$. 
From $\mb{y} \neq \mb{0}$ and Lemma \ref{l_beta_entry_apart_zero} we find a $j$ such that $\beta_j \neq 0$.
For notational convenience, assume that $j = 1$.
Thus, we can write $\mb{x}_1 = \sum_{i \neq 1} \beta_i  \beta_1^{-1} \mb{x}_i -\beta_1^{-1} \mb{y}$. \newline
Let $Z$ be the set of vectors of $X$ but with $\mb{x_1}$ swapped for $\mb{y}$.
The expression for $\mb{x}_1$ proves $\mb{x}_1 \in \Span{Z}$ and $\Span{X} \subseteq \Span{Z}$ immediately follows.
Because $\mb{y}$ is a linear combination of $\mb{x_1},...,\mb{x_n}$, any linear combination of vectors from $Z$ is a linear combination of vectors of $X$.
This shows $\Span{Z} \subseteq \Span{X}$ and hence $\Span{X} = \Span{Z}$.
\end{proof}

\begin{lemma}[Independent Austauschlemma] \label{l_independent_austausch}
Let $X$, $Z$ and $\mb{y}$ be as in the Austauschlemma \ref{l_austausch}.
Moreover, suppose that the vectors in $X$ are independent.
Then $Z$ is also a set of independent vectors.
\end{lemma}
\begin{proof}
Notation as in the proof of Lemma \ref{l_austausch}.
We need to show that $Z = \{ \mb{z}_1 = \mb{y}, \mb{z}_2 = \mb{x}_2,...,\mb{z}_n = \mb{x}_n \}$ consists of independent vectors.
Let $\gamma \in \reals^m$ be apart from $\mb{0}$.
Using $\mb{z}_1 = \mb{y} = \sum_i \beta_i \mb{x}_i$, we can write $\sum_i \gamma_i \mb{z}_i = \sum_i \alpha_i \mb{x}_i$, where $\alpha_i = \gamma_i + \gamma_1  \beta_i$ if $i \neq 1$ and $\alpha_1 = \gamma_1  \beta_1$.
From $\gamma \neq \mb{0}$ we know either $\gamma_1 \neq 0$ or there is some $j \neq 1$ such that $\gamma_j \neq 0$.
In the first case $\alpha_1 = \gamma_1  \beta_1 \neq 0$ and hence $\alpha \neq \mb{0}$.
By the independence of $\mb{x}_1,...,\mb{x}_n$ this establishes $\sum_i \alpha_i \mb{x}_i = \sum_i \gamma_i \mb{z}_i \neq \mb{0}$.
In the second case, we look at $\alpha_j = \gamma_j + \gamma_1  \beta_j$.
We know $\gamma_j \neq 0$, so either $\gamma_1  \beta_j \neq -\gamma_j$ or $\gamma_1  \beta_j \neq 0$ by Lemma \ref{l_apartness_cotransitive}.
The first of these two cases implies $\alpha_j \neq 0$, which again yields $\alpha \neq \mb{0}$.
Similarly $\gamma_1  \beta_j \neq 0$ implies $\gamma_1 \neq 0$, which brings us back to our first case.
\end{proof}

\begin{theorem}[Steinitz Austauschsatz] \label{t_steinitz}
Let $X = \{ \mb{x}_1,...,\mb{x}_n \}$ be a set of independent vectors.
Suppose $Y = \{ \mb{y}_1,...,\mb{y}_m \}$ is a set of vectors spanning $X$.
Then $n \leq m$ and there is a subset $Z = \{ \mb{z_1},...,\mb{z_{m - n}} \}$ of $Y$ such that $\Span ( X \cup Z ) = \Span Y$.
\end{theorem}
\begin{proof}
Successively apply Lemma \ref{l_austausch}, exchanging $\mb{x}_1$ for some $\mb{y} \in Y$ yielding $Y_1$ with $\Span Y_1 = \Span Y$, exchanging  $\mb{x}_2$ for some  $\mb{y} \in Y_1$ yielding $Y_2$ with $\Span Y_2 = \Span Y$, etcetera.
Note that the role of $X$ in the notation of Lemma \ref{l_austausch} is played by $Y$.
We obtain a set of vectors $A$ consisting of $\min(n, m)$ vectors from $X$ and $m - \min(n, m)$ vectors from $Y$.
This set $A$ has the same span as $Y$.
Next we show $n \leq m$, we will argue by contradiction.
If we have $n > m$, then $A = \{ \mb{x}_1,...,\mb{x}_m \}$.
Since $\Span X \subseteq \Span Y = \Span A$, we see in particular that $\mb{x}_n \in \Span A$.
Clearly this contradicts the independence of $\{ \mb{x}_1,...,\mb{x}_n \}$.
Thus $n \leq m$ and we define $Z$ as the $m - n$ vectors from $Y$ in $A$.
\end{proof}
\begin{proposition} \label{p_basis_def}
Let $X = \{ \mb{x}_1,...,\mb{x}_n \}$ be a set of independent vectors.
Suppose $Y = \{ \mb{y}_1,...,\mb{y}_m \}$ is a set of vectors spanning $X$.
If $n = m$, then $Y$ is also independent and $\Span X = \Span Y$.
\end{proposition}
\begin{proof}
We get $\Span X = \Span Y$ as an immediate consequence of Theorem \ref{t_steinitz} and $n = m$.
To prove the independence of $Y$, we will successively apply Lemma \ref{l_independent_austausch}.
We want to exchange the vectors in $X$ for those in $Y$.
That means we first have to establish $\mb{y} \neq \mb{0}$ for all $\mb{y} \in Y$.
We will just establish $\mb{y_1} \neq \mb{0}$, which suffices by symmetry.
From $\Span Y = \Span X$ we know that there exists a $\beta \in \reals^n$ such that $\mb{y}_1 = \sum_i \beta_i \mb{x}_i$.
Apply Lemma \ref{l_austausch} successively, replacing the vectors in $Y$ with those from $X$.
This is allowed because $X$ is independent and therefore $\mb{x} \neq \mb{0}$ for all $\mb{x} \in X$.
We obtain $\Span \{ \mb{x}_1,...,\mb{x}_{n-1}, \ \mb{y}_1 \} = \Span Y$.
Since we have already established $\Span X = \Span Y$, there is an $\alpha \in \reals^n$ such that $\mb{x}_n = \sum_{i < n} \alpha_i \mb{x}_i + \alpha_n \mb{y}_1$.
Substituting our expression for $\mb{y}_1$ in this last equation, we have written $\mb{x}_n$ as a linear combination of the vectors $\mb{x}_1,...,\mb{x}_n$.
By Lemma \ref{l_weakly_indep_equi_unique_repr}, all coefficients must be $0$, except that of $\mb{x}_n$.
This coefficient is obviously $1$, so $\alpha_n \beta_n = 1$ and in particular $\beta_n \neq 0$.
Therefore, we see $\beta \neq \mb{0}$ and again by the independence of $\mb{x}_1,...,\mb{x}_n$ we now know $\mb{y}_1 \neq \mb{0}$.
\newline
Thus we are allowed to successively apply Lemma \ref{l_independent_austausch}.
We exchange the vectors in $X$ for those in $Y$.
Because $n = m$ this will result in us getting $Y$ back again.
Since $X$ is independent and the exchange of vectors in Lemma \ref{l_independent_austausch} preserves independence, $Y$ is also independent. 
\end{proof}

We call a finite set of vectors $X = \{ \mb{x}_1,...,\mb{x}_n \}$ a \textit{basis} of a vector space $V$ if $V = \Span X$ and $X$ is independent.
By Proposition \ref{p_basis_def} any two bases have the same number of elements.
Thus, we can define the \textit{dimension} of a vector space as the cardinality of one of its bases.
Note that not every vector space has a dimension.
A trivial example is $\Span \{ \lvert \rho \rvert \}$, where $\rho = 0$ is undecidable.
Although independence and dimension are generally undecidable, there is something we can say about arbitrary collections of vectors.
This requires us to recall some results on angles between vectors and hence projections of vectors.
\clearpage
\begin{definition}
For vectors $\mathbf{x} \in \mathbb{R}^n$ the \textbf{Euclidean} or \textbf{$L^2$ norm} $\norm{\mathbf{x}}$ of $\mathbf{x}$ is defined as $\sqrt{\sum_{i = 1}^n \mathbf{x}_i^2}$, as usual.    
\end{definition}

\begin{definition}\label{d_proj}
Let $X$ be a linear subspace of $\mathbb{R}^m$ and $\mb{y}$ be a vector in $\mathbb{R}^m$.
A \textbf{projection} of $\mb{y}$ onto $X$ is a vector $proj_X \mb{y}$ satisfying:
\begin{enumerate}
    \item $proj_X \mb{y}$ lies in $\Span X$,
    \item $\mb{y} - proj_X \mb{y}$ is orthogonal to $X$.
\end{enumerate}
\end{definition}

\begin{lemma}\label{l_proj}
Let $X$ be a linear subspace of $\mathbb{R}^m$ and $\mb{y}$ be a vector in $\mathbb{R}^m$.
If $X$ has a basis $B$, then $proj_X \mb{y}$ exists and is independent of the chosen basis.
\end{lemma}
\begin{proof}
Using the usual Gram-Schmidt method applied to the basis $B = \{ \mb{b}_1,...,\mb{b}_n \}$, we get an orthonormal basis $U= \{ \mb{u}_1,...,\mb{u}_n \}$ of $X$.
This method is constructive, as the orthogonal basis vectors defined by the iterative formulation $\mb{u}'_1 = \mb{b}_1$ and $\mb{u}'_{k + 1} = \mb{b}_{k + 1} - \sum_{i \leq k} \langle \mb{u}'_i, \mb{b}_{k + 1} \rangle \langle \mb{u}'_i, \mb{u}'_i \rangle^{-1} \mb{u}'_i$.
The only issue could be in dividing by $\langle \mb{u}'_i, \mb{u}'_i \rangle$, since we then need $\norm{\mb{u}'_i} > 0$.
However, the vectors $\mb{u}'_i$ can be rewritten as linear combinations of the original independent vectors $\mb{b}_i$. 
Moreover, the coefficient of $\mb{b}_i$ in $\mb{u}'_i$ is always apart from $0$.
Because $B$ is a basis, this implies $\norm{\mb{u}'_i} > 0$.
We can then create the orthonormal basis vectors $\mb{u}_i = \mb{u}'_i / \norm{\mb{u}'_i}$ and define $proj_X \mb{y}$ with the usual formula $\sum_{i \leq n} \langle\mb{y}, \mb{u}_i\rangle \mb{u}_i$.

Next, we prove that the projection is uniquely determined.
The proof is identical to the classical version.
Since any $proj_X \mb{y}$ is in the span of $U$, we may write it as $\sum_{i \leq n} \alpha_i \mb{u}_i$.
Taking the inner product of this expression and $\mb{u}_i$ gives us $\alpha_i = \langle proj_X \mb{y}, \mb{u}_i\rangle$ by the orthogonality of the vectors $\mb{u}_i$.
At the same time, we can show that the expression $\langle proj_X \mb{y}, \mb{u}_i\rangle$ equals $\langle \mb{y}, \mb{u}_i \rangle$.
Because $\mb{y} - proj_X \mb{y}$ is orthogonal to $X$, it is also orthogonal to the basis vectors in $U$.
Writing $\langle \mb{y} - proj_X \mb{y}, \mb{u}_i \rangle = 0$ then yields $\langle \mb{y}, \mb{u}_i \rangle = \langle proj_X \mb{y}, \mb{u}_i \rangle$.
Thus, the coefficient vector $\alpha$ is uniquely determined.
\end{proof}

For an arbitrary collection of vectors $X$, Theorem \ref{t_classification} seems the strongest possible positive result on independence.
We can approximate the dimension of $X$ from below and up to vectors of small length or vectors in or nearly in the span of currently identified independent columns. 
Theorem \ref{t_classification} distinguishes these two, not necessarily distinct, cases of failure to prove independence.
We first derive two lemmata to obtain this theorem.

\begin{lemma}\label{l_proj_indep}
Let $X = \Span \{\mb{x}_1,...,\mb{x}_n \} $ be an independent linear subspace of $\mathbb{R}^m$ and $\mb{y}$ a vector in $\mathbb{R}^m$.
Then $y$ is independent from $X$ if and only if $\norm{\mb{y} - proj_X \mb{y}} > 0$.
\end{lemma}
\begin{proof}
This proof is identical to the classical variant.
Suppose $\mb{y}$ is independent from $X$.
By Definition \ref{d_proj} the projection $proj_X \mb{y}$ is a linear combination of vectors in $X$.
Then Definition \ref{d_indep} precisely states $\mb{y} \neq proj_X \mb{y}$ and hence $\mb{y} - proj_X \mb{y} \neq 0$ and $\norm{\mb{y} - proj_X \mb{y}} > 0$.
Next, assume $\norm{\mb{y} - proj_X \mb{y}} > 0$ and let $\sum_i \beta \mb{x}_i$ be an arbitrary linear combination of vectors from $X$.
Because $\mb{y} - proj_X \mb{y}$ is orthogonal to $X$, we can write:

\begin{align*}
    \norm{\mb{y} - \sum_i \beta \mb{x}_i}^2 &= \norm{\mb{y} - proj_X \mb{y} + proj_X \mb{y} - \sum_i \beta \mb{x}_i}^2 \\
    &= \norm{\mb{y} - proj_X \mb{y}}^2 + \norm{proj_X \mb{y} - \sum_i \beta \mb{x}_i}^2 \\
    &\geq \norm{\mb{y} - proj_X \mb{y}}^2 \\
    &> 0.
\end{align*}
Thus, we have established $\norm{\mb{y} - \sum_i \beta \mb{x}_i} >0$ and therefore $\mb{y}$ is independent of $X$.
\end{proof}

\begin{lemma}\label{l_trich}
Let $X$ be some independent linear subspace of $\mathbb{R}^m$.
Then for every vector $\mb{y} \in \mathbb{R}^m$ and $\varepsilon > 0$, we can decide between the following alternatives:
\begin{enumerate}
    \item $\mb{y}$ is independent from $X$,
    \item $\norm{\mb{y}} < \varepsilon$,
    \item the angle between $\mb{y}$ and $X$ is less than $\varepsilon$.
\end{enumerate}
\end{lemma}   
\begin{proof}
First, decide $\norm{\mb{y}} < \varepsilon \lor \norm{\mb{y}} > 0$, in the first case we are done.
If $ \norm{\mb{y}} > 0$, calculate $proj_X(\mb{y})$ using Lemma \ref{l_proj} and consider $\mb{a} =\mb{y} - proj_X(\mb{y})$.
We again distinguish two cases by considering $\norm{\mb{a}}$.
If $\norm{\mb{a}} > 0$, then $\mb{y}$ is independent of $X$ by Lemma \ref{l_proj_indep}.
Finally, consider the case $\norm{\mb{a}} < \sin(\min(\varepsilon, \frac{1}{2}\pi))  \norm{\mb{y}}$.
This decision is allowed because $\sin(\min(\varepsilon, \frac{1}{2}\pi)))  \norm{\mb{y}} > 0$.
Now, the angle between $\mb{y}$ and $X$, which equals $\sin^{-1}(\norm{\mb{a}} / \norm{\mb{y}})$, is less than $\varepsilon$.
\end{proof}

\begin{theorem}\label{t_classification}
Let $X = \Span \{ \mb{x}_1,...,\mb{x}_n \}$, then for each $\varepsilon > 0$ the vectors $\mb{x}_1,...,\mb{x}_n$ can be divided into three groups:
\begin{enumerate}
    \item a set $B$ consisting of independent vectors,
    \item a set $\tilde{B}$ of vectors whose angle with $B$ is smaller than $\varepsilon$,
    \item a set $E$ of vectors with $\norm{\mb{e}} < \varepsilon$ for all $\mb{e} \in E$.

\end{enumerate}
\end{theorem}
\begin{proof}
Iteratively apply lemma \ref{l_trich}, obviously starting with $B = \emptyset$.
\end{proof}
\section{Gauss-Jordan elimination} \label{sec_gauss}
We denote matrices with capital letters such as $A$ and $B$.
With $[A]_i$ we denote the $i$'th column of the matrix $A$.
The vector space generated by the columns of $A$ is called the \textit{column space} of $A$.
The algebra of matrices obviously presents no constructive difficulties.
In general, not much of the usual results on matrices hold constructively, unless we restrict ourselves to \textit{ranked} matrices, which we define below.

\begin{definition} \label{d_rank}
Let $A$ be a matrix.
The span of the columns of $A$ is denoted by $\Col A$.
If $\Col A$ has dimension $n$, then we say that $A$ is a \textbf{ranked} matrix and its rank is $n$.
We will also write $\rank A = n$.
If we write an expression like $\rank A < m$ for some $m \in \mathbb{N}$, then this implicitly states that $\rank A$ exists.
\end{definition}

Since $\rank A = n$ means there exists \textit{a} basis of $\Col A$ of dimension $n$, it is not immediately obvious that we can choose $n$ columns of $A$ of $\Col A$ as a basis.
Fortunately this is precisely what the independent Austauschlemma \ref{l_independent_austausch} proves.

\begin{lemma} \label{l_matrix_basis_from_cols}
If $A$ is a matrix with $\rank A = n$, then there is a basis for $\Col A$ consisting of $n$ columns of $A$.
\end{lemma}
\begin{proof}
Successively apply Lemma \ref{l_independent_austausch}, swapping vectors from a basis of the column span of $A$ for columns of $A$.
\end{proof}

The previous lemma is often omitted from classical treatments of linear algebra.
Most likely because its proof is trivial.
The statement mentions the rank of a matrix.
It must therefore be preceded by the statement that any two bases have equal cardinality, in both classical and constructive treatments.
The `implicit classical proof' of Lemma \ref{l_matrix_basis_from_cols} then likely goes as follows.
Go over each of the columns of $A$.
As long as the current column is not included in the span of previously collected columns, add it to this collection.
The final collection of columns is clearly linearly independent and spans $\Col A$.
Therefore, it must be a basis of $\Col A$.
Because bases have equal numbers of elements, apparently we have collected $\rank A$ columns with our algorithm.
Obviously, we cannot use this algorithm in constructive mathematics.
In general, we cannot decide $\mb{x} \in \Span(\{ \mb{x}_1, \mb{x}_2, ...,\mb{x}_n \})$.
Therefore, Lemma \ref{l_matrix_basis_from_cols} deserves a proof.
That it is still a trivial proof is due to Lemma \ref{l_independent_austausch}.
This strengthening of the Steinitz Austauschlemma with preservation of independence is often left out in classical treatments.
Although still true, it is simply not needed.
We do not need it to prove the Austauschsatz, which in turn suffices to prove the well-definedness of dimension.
This is the same constructively, but the regular Austauschlemma does not directly imply Lemma \ref{l_matrix_basis_from_cols}.

Matrices were invented to solve systems of linear equations.
This is still done through the process of \textit{Gauss-Jordan elimination}.
We successively apply simple operations (Definition \ref{d_elementary_operations}) to a matrix of coefficients and the vector of outcomes.
This results in a matrix and vector whose corresponding linear system has the same solutions as our original equations.
However, the operations were chosen in such a way that the resulting system is easily solved.
The matrix of coefficients has been reduced to a \textit{reduced row echelon form}.

\begin{definition}[Elementary operations] \label{d_elementary_operations}
The \textbf{elementary row operations} on the rows of a matrix are:
\begin{enumerate}
\item exchanging two rows,
\item multiplying a row with a number apart from $0$,
\item adding a multiple of a row to another row.
\end{enumerate}
The \textbf{elementary column operations} are defined analogously.
\end{definition}

We first show that these operations do not change the rank of a ranked matrix in Proposition \ref{p_elem_col_operations_preserve_col_span} and Proposition \ref{p_elem_row_operations_preserve_col_rank}.
To this end we first explain the notion of \textit{strong extensionality} and then prove two supporting lemmata in Lemma \ref{l_linear_map_basis_to_basis} and Lemma \ref{l_elem_row_are_extensional}.

\begin{definition} \label{d_strongly_extensional}
A linear map $T$ from $\mathbb{R}^n$ to $\mathbb{R}^m$ is called \textbf{strongly extensional} if $\mb{x} \neq  \mb{0}$ implies $T(\mb{x}) \neq  \mb{0}$.
\end{definition}

Obviously, this definition of strong extensionality holds more generally.
It applies to all mappings between spaces with so-called \textit{apartness} relations.
These are generalizations of the apartness of real numbers.
Since we are just looking at linear algebra, we do not need to consider this.

\begin{lemma} \label{l_linear_map_basis_to_basis}
Let $B = \{ \mb{b}_1,...,\mb{b}_n \}$ be a basis of the vector space $V$.
Suppose $T$ is a strongly extensional, linear transformation $V \rightarrow W$.
Then $T(B) = \{ T(\mb{b}_1),...,T(\mb{b}_n) \}$ is a basis of $T(V)$.
\end{lemma}
\begin{proof}
We need to prove $T(B)$ spans $T(V)$ and consists of independent vectors.
First we show that $T(B)$ spans $T(V)$, so let $\mb{v} \in V$.
Because $\Span B = V$ we can write $\mb{v}$ as $\sum_i \beta_i \mb{b}_i$ for some $\beta \in \mathbb{R}^n$.
From linearity of $T$ we see $T(\mb{v}) = \sum_i \beta_i T(\mb{b}_i) \in \Span T(B)$.
Next, we prove $\{ T(\mb{b}_1),...,T(\mb{b}_n) \}$ is independent.
Let $\beta \neq 0$ and consider $\sum_i \beta_i T(\mb{b}_i) = T( \sum_i \beta_i \mb{b}_i)$.
Since $B$ is a basis we see $\sum_i \beta_i \mb{b}_i \neq \mb{0}$. 
Because $T$ is strongly extensional this implies $T( \sum_i \beta_i \mb{b}_i) \neq  \mb{0}$.
Thus $T(B)$ consists of independent vectors.
\end{proof}

\begin{lemma} \label{l_elem_row_are_extensional}
Let $A$ be an arbitrary matrix.
Every elementary row operation is a strongly extensional, linear map $T$ from $\Col A$ to $T(\Col A)$.
\end{lemma}
\begin{proof}
Linearity is obvious for each of the row operations.
Consider swapping two rows.
If $\mb{x} \neq \mb{0}$, then there is an index $i$ such that $x_i \neq 0$.
The vector $T(\mb{x})$ is then clearly also apart from $\mb{0}$, there is an index $j$ such that $T(\mb{x})_j \neq 0$.
Next, we look at multiplying a row with a scalar $\lambda$ apart from $0$.
By symmetry assume we multiply the first row with $\lambda$.
If $\mb{x} \neq \mb{0}$ then there is an index $i$ such that $x_i \neq 0$.
Since $x_i = T(\mb{x})_i$ if $i > 1$ in this case, we only have to consider what happens when $i = 1$.
We get $T(\mb{x})_1 = \lambda x_1 \neq 0$.
Finally, consider adding a multiple of a row to another row.
By symmetry suppose that we add $\lambda$ times the second row to the first row.
If $\mb{x} \neq \mb{0}$ then there is an index $i$ such that $x_i \neq 0$.
Since $x_i = T(\mb{x})_i$ if $i > 1$ in this case, we only have to consider what happens when $i = 1$.
From $x_1 \neq 0$ we can decide $\lambda x_2 \neq - x_1$ or $\lambda x_2 \neq 0$.
In the first case we see $x_1 + \lambda x_2 = T(\mb{x})_1 \neq 0$.
In the second case we have $x_2 \neq 0$, which again shows $T(\mb{x})_2 \neq 0$.
\end{proof}

\begin{proposition}
\label{p_elem_row_operations_preserve_col_rank}
Let $A$ be a ranked matrix.
If $A'$ is obtained from $A$ by applying elementary row operations, then $\rank A' = \rank A$.
\begin{proof}
By the previous proposition, every elementary row operation is a strongly extensional, linear map. 
Choose a basis $\{ \mb{x}_1,...,\mb{x}_r \}$  of $\Col A$ consisting of $r = \rank A$ columns from $A$.
Applying a single elementary row operation $T$ to the matrix $A$ we get the matrix $B$.
By Lemma \ref{l_linear_map_basis_to_basis} $\Col B$ has a basis consisting of the $r$ vectors $\{ T(\mb{x}_1),...,T(\mb{x}_r) \}$.
This shows $\rank B = \rank A$. 
Repeating this argument for every elementary row operation that transforms $A$ into $A'$, we prove the desired result.
\end{proof}
\end{proposition}

\begin{proposition}
\label{p_elem_col_operations_preserve_col_span}
Let $A$ be a matrix.
If $B$ is obtained from $A$ by applying elementary column operations, then $\Col B = \Col A$.
\end{proposition}
\begin{proof}
We reason inductively, showing that each of the three elementary operations does not affect the column space.
Clearly swapping columns does not change the column space of a matrix.
Write $\mb{a}_1, \mb{a}_2,...,\mb{a}_n$ for the $n$ columns of $A$ and define $\mb{b}_1, \mb{b}_2,...,\mb{b}_n$ similarly.
We secondly show that multiplying a column with a scalar apart from $0$ does not affect the column space.
Suppose $\mb{x} = \sum_i \beta_i \mb{a}_i \in \Col A$ and $B$ is obtained by multiplying column 1 with $\lambda \neq 0$, so $\mb{b}_1 = \lambda \mb{a}_1$.
Then we can write $\mb{x} = \sum_{i \neq 1} \beta_i \mb{a}_i + \lambda^{-1}  \beta_1 \mb{b}_1$, showing $\mb{x} \in \Col B$.
Likewise $\sum_i \beta_i \mb{b}_i = \sum_{i \neq 1} \beta_i \mb{a}_i + \lambda  \beta_1 \mb{a_1}$, which proves $\Col B \subseteq \Col A$.
Finally, we prove that adding a multiple of a column to another column does not alter the column space.
Suppose $\mb{x} = \sum_i \beta_i \mb{a}_i \in \Col A$ and $B$ is obtained by adding $\kappa$ times column 1 to column 2.
This means $\mb{b}_2 = \mb{a}_2 + \kappa \mb{a}_1$.
Suppose $\mb{x} = \sum_i \beta_i \mb{a}_i \in \Col A$, then we can write $\mb{x} = \sum_{i \neq 2} \beta_i \mb{b}_i + \beta_2 \mb{b}_2 - \kappa  \beta_2 \mb{b}_1$.
This shows $\Col A \subseteq \Col B$.
We see $\Col B \subseteq \Col A$ from $\sum_i \beta_i \mb{b}_i = \sum_i \beta_i \mb{a}_i + \kappa  \beta_2 \mb{a}_1$.
\end{proof}

Next we define the \textit{row echelon form} and \textit{reduced row echelon form} of a matrix exactly as in classical mathematics in Definition \ref{d_row_echelon_form} and \ref{d_reduced_row_echelon_form}.
In general, we see that not every matrix can be reduced to this latter form.
This statement is equivalent to an omniscience principle (Proposition \ref{p_partial_pivoting_lpo}).
For ranked matrices, we obtain a constructive version of Gauss-Jordan elimination in Theorem \ref{t_gauss}.

\begin{definition} \label{d_row_echelon_form}
A matrix $A$ is in \textbf{row echelon form} if:
\begin{enumerate}
\item all rows either start with zero or more $0$'s, followed by an element apart from $0$, or is the zero row,
\item if a row has more initial zeroes than another row, it is below the row with less zeroes, 
\item two rows can only have the same number of initial zeroes if they are identically $\mathbf{0}$.
\end{enumerate}

Such a matrix for instance looks like this:
\begin{equation*}
\begin{vmatrix}
0 & + & . & . & . & .\\
0 & 0 & 0 & + & . & .\\
0 & 0 & 0 & 0 & + & .\\
0 & 0 & 0 & 0 & 0 & 0\\
0 & 0 & 0 & 0 & 0 & 0\\
\end{vmatrix}.
\end{equation*}
Here, $+$ stands for an entry apart from $0$ and $.$ stands for arbitrary entry.
\end{definition}

\begin{definition} \label{d_reduced_row_echelon_form}
A matrix $A$ is in \textbf{reduced row echelon form} if:
\begin{enumerate}
\item it is in row echelon form,
\item the first entry in a row apart from $0$ is always a $1$, 
\item all elements above such initial ones are $0$.
\end{enumerate}

Such a matrix for instance looks like this:
\begin{equation*}
\begin{vmatrix}
0 & 1 & . & 0 & 0 & .\\
0 & 0 & 0 & 1 & 0 & .\\
0 & 0 & 0 & 0 & 1 & .\\
0 & 0 & 0 & 0 & 0 & 0\\
0 & 0 & 0 & 0 & 0 & 0\\
\end{vmatrix}.
\end{equation*}
Here, $.$ stands for an arbitrary entry.
\end{definition}

\begin{proposition} \label{p_partial_pivoting_lpo}
Let $T$ be the statement
`Let $A$ be a ranked matrix with at least two columns.
By applying elementary row operations,
$A$ can be transformed into reduced row echelon form.'
Then $T$ is equivalent to LPO.
\end{proposition}
\begin{proof}
Let $\rho \in \reals$ and consider the $1 \times 2$ matrix $\lvert \rho \ 1 \rvert$.
Obviously, this matrix has rank $1$.
Suppose that we can apply elementary row operations to reduce this matrix to reduced row echelon form.
As it is a $1 \times 2$ matrix, we can only apply the multiplication with a number apart from $0$ row operation.
We then get a matrix $\lvert \beta \rho \  \beta \rvert$, where $\beta \neq 0$.
This can only be in reduced row echelon form if $\beta \rho = 0 \lor \beta \rho = 1$.
The first case proves $\rho = 0$ and the second case proves $\rho \neq 0$.
Thus, we have derived LPO.

Next, assume LPO, we shall derive $T$.
Let $A$ be a ranked matrix, then for each entry we can decide if it equals $0$ or is apart from $0$.
Clearly, this is sufficient to make the usual algorithm of Gauss-Jordan elimination work.
\end{proof}

One can wonder if the previous proposition can be simplified.
Can the ranked $1 \times 1$ matrices be reduced to row echelon form?
This turns out to be true and is not equivalent to some omniscience principle.
Clearly, the $1 \times 1$ matrix $\lvert \rho \rvert$ is ranked if and only if $\rho = 0  \lor \rho \neq 0$.
Thus, the undecidability of row reducing ranked matrices is a `linear algebraic' feature of constructive linear algebra.
The problem does not arise solely from the undecidability of real number ordering.
On the other hand, the proof of Proposition \ref{p_partial_pivoting_lpo} uses the fact that $1 \times n$ matrices allow only for a single kind of elementary row operation.
This is not an essential part of the statement, as the following proposition shows.
For larger matrices where more row operations are available, Gauss-Jordan elimination is still undecidable.

\begin{proposition}
Let $T$ be the statement 
`Let $A$ be a ranked $2 \times 2$ matrix.
By applying elementary row operations,
$A$ can be transformed into reduced row echelon form.'
Then $T$ is equivalent to LPO.
\end{proposition}
\begin{proof}
We first show that $T$ proves LPO.
Consider the following matrix $A$ with rank $1$:
\begin{equation*}
\begin{vmatrix}
\rho & 1\\
0 & 0\\
\end{vmatrix}.
\end{equation*}
After applying elementary row operations, we always end up with a matrix of the form:
\begin{equation*}
\begin{vmatrix}
\lambda \rho & \lambda \\
\kappa \rho & \kappa \\
\end{vmatrix}.
\end{equation*}
Here we have $\lambda \neq 0 \lor \kappa \neq 0$.
This can easily be shown inductively.
Initially, the matrix $A$ has this form and all three elementary row operations preserve this shape.
This is trivial for swapping rows and multiplying a row with a number apart from $0$.
We only prove the assertion for addition of rows.
By symmetry, it suffices to consider adding the first row to the second row.
This still leaves two cases, $\lambda \neq 0$ or $\kappa \neq 0$.
If $\lambda \neq 0$, then the preservation of the form is again trivial.
So suppose $\kappa \neq 0$ and we add $\beta$ times the first row to the second row.
We then get the following matrix:
\begin{equation*}
\begin{vmatrix}
\lambda \rho & \lambda \\
(\kappa + \beta \lambda) \rho & \kappa + \beta \lambda \\
\end{vmatrix}.
\end{equation*}
Since we know $\kappa \neq 0$, we can decide $\beta \lambda \neq -\kappa \lor \beta \lambda \neq 0$.
The first case proves that $\kappa + \beta \lambda$ is apart from $0$.
In the second case, we see $\lambda \neq 0$.
Either way, the resulting matrix has the postulated form.

Thus, after an arbitrary sequence of elementary row operations, the first column of our matrix is $(\lambda \rho \ \kappa \rho)^T$, where $\lambda \neq 0 \lor \kappa \neq 0$.
If the matrix is in reduced row echelon form, this implies $\lambda \rho = 0 \lor \lambda \rho = 1$ as well as $\kappa \rho = 0$.
The case $ \lambda \rho = 1$ proves $\rho \neq 0$.
If $\lambda \rho = 0$, then combining this with $\kappa \rho = 0$ and $\lambda \neq 0 \lor \kappa \neq 0$ we get $\rho = 0$.
Thus, we can decide $\rho = 0 \lor \rho \neq 0$ for an arbitrary real number $\rho$.
This is equivalent to LPO.

Proving LPO implies $T$ is again trivial.
Ordering of real numbers is decidable if LPO holds.
The usual process of Gauss-Jordan elimination can then be applied.
\end{proof}

Before we get to the positive, constructive version of the classical notion of Gauss-Jordan elimination, we need to derive a small lemma.

\begin{lemma} \label{l_apart_0_mat}
Let $A$ be a ranked matrix, then $\rank A > 0$ if and only if $A$ has an entry apart from $0$.
\end{lemma}
\begin{proof}
Suppose $\rank A > 0$ and pick a basis of $\Col A$ from the columns of $A$.
This basis contains at least one column $\mb{x}$.
Since the set $\{ \mb{x} \}$ is independent, we see that $1 \mb{x} \neq \mb{0}$ and thus $A$ has an entry apart from $0$.
Next assume $A$ has an entry apart from $0$.
Because $A$ is ranked, it suffices to derive a contradiction from $\rank A = 0$.
This is clearly contradictory because $\rank A = 0$ is equivalent to $\Col A = \Span \{ \mb{0} \}$.
\end{proof}

\begin{theorem} \label{t_gauss}
Let $A$ be a ranked matrix.
By applying elementary row operations and swapping columns,
$A$ can be transformed to reduced row echelon form.
\end{theorem}
\begin{proof}
We prove this by induction on the rank of $A$.
If $\rank A = 0$, then clearly all entries of $A$ are zero and we are done.
Suppose that we have shown the result for $\rank A = n$.
So we consider a matrix with rank $n + 1$.
In this case $\rank A > 0$, so there exists an element apart from $0$ from Lemma \ref{l_apart_0_mat}.
By swapping rows and columns, we can move this element to the first row and column.
Next we apply row multiplications and additions to transform the first column into $(1 \ 0 \ ... \ 0)^T$.
Call the resulting matrix $A'$.
By Proposition \ref{p_elem_row_operations_preserve_col_rank} and \ref{p_elem_col_operations_preserve_col_span} we see $\rank A = \rank A' > 0$.
Let $B'$ be a basis of $\Col A'$ consisting of columns of $A'$, one of which is the first column $\mb{a}'_1 = (1 \ 0 \ ... \ 0)^T$.
This is possible using Lemma \ref{l_independent_austausch}.
If this is in fact the only column of $B'$ then we are done.
In this case all columns of $A'$ are in the span of $\mb{a}'_1$ and $A'$ is in reduced row echelon form.
So suppose $B'$ consists of more vectors and write $B' = \{\mb{b}'_0 = \mb{a}'_1, \mb{b}'_1,...,\mb{b}'_n \}$.
Consider the submatrix $A''$ of $A'$ obtained by dropping the first row and column of $A'$.
Clearly the column span of $A''$ is spanned by the vectors $B''$ consisting of those in $B'$ with their first elements removed and excluding $\mb{a}_1'$.
Write $B'' = \{\mb{b}''_1,...,\mb{b}''_n \}$, where obviously $\mb{b}''_i$ is $\mb{b}'_i$ with its  first element removed.

Next, we prove that $B'' = \{\mb{b}''_1,...,\mb{b}''_n \}$ is independent and hence a basis of $A''$.
Let $\beta \in \mathbb{R}^n$ be apart from $\mb{0}$ and look at $\sum_{i > 0} \beta_i \mb{b}''_i$.
As $B'$ is a basis we know $\sum_{i \geq 0} \beta_i \mb{b}'_i \neq \mb{0}$.
Find a $j$ such that $\big( \sum_{i \geq 0} \beta_i \mb{b}'_i \big)_j \neq 0$.
If $j > 1$ then obviously $\sum_{i > 0} \beta_i \mb{b}''_i$ is apart from $\mb{0}$.
This holds because $\mb{b}'_0 = \mb{a}_1 = (1 \ 0 \ ... \ 0)^T$ and hence $\big( \sum_{i \geq 0} \beta_i \mb{b}'_i \big)_j = \big( \sum_{i > 0} \beta_i \mb{b}''_i \big)_j$.
So suppose $j = 1$ and consider the sum $\sum_{i \geq 0} \beta_i \mb{b}'_i - \big( \sum_{i \geq 0} \beta_i \mb{b}'_i \big)_1 \mb{a}'_1$.
We may also write this as $\sum_{i \geq 0} \gamma_i \mb{b}'_i$, where $\gamma_i = \beta_i$ if $i > 0$ and $\gamma_0 = - \big( \sum_{i > 0} \beta_i \mb{b}'_i \big)_1$.
Obviously, the vector $\gamma$ is apart from $\mb{0}$ and then so is  $\sum_{i \geq 0} \gamma_i \mb{b}'_i$ from the independence of $B'$.
As the first entry of $\sum_{i \geq 0} \gamma_i \mb{b}'_i$ is $0$, there must be a $k > 1$ such that $\big(\sum_{i \geq 0} \gamma_i \mb{b}'_i\big)_k \neq 0$.
Again we see $\big(\sum_{i \geq 0} \gamma_i \mb{b}'_i\big)_k = \big( \sum_{i > 0} \beta_i \mb{b}''_i \big)_k$.
Thus also in the case $j = 1$ we can find a $k$ such that $\big( \sum_{i > 0} \beta_i \mb{b}''_i \big)_k$ is apart from $0$.

We have established that $B''$ is a basis of $A''$ and hence $A''$ is a matrix with rank $n$.
Thus we can reduce $A''$ to reduced row echelon form.
Now consider what happens if we apply the row and column operations used to reduce $A''$ to the matrix $A'$.
The first column of $A'$ will not be affected.
Indeed, the very first element $A'_{1,1}$ will be completely untouched, and the row operations will only work on the other elements of the first column of $A'$.
However, as all these entries are $0$, this will not change this column.
We now have a matrix whose first column is $(1 \ 0 \ ... \ 0)^T$ and where the submatrix consisting of all but the first row is in reduced row-echelon form.
Thus by subtracting suitable multiples of leading one rows from the first row, we are done.
\end{proof}

Note that we refrained from talking about \textit{the} reduced row echelon form of a matrix.
Because we allow column swapping in reducing matrices, the Gauss-Jordan procedure will not produce a unique reduced row echelon form.
The best we can say is that the number of leading one columns is unique. 
This is simply the rank of the matrix.

Swapping columns also means it is not directly clear how we can use Theorem \ref{t_gauss} to solve systems of linear equations.
Consider the system $A\mb{x} = \mb{b}$ with augmented matrix $A | \mb{b}$.
If the rank of both $A$ and $A | \mb{b}$ is known, then we can solve our system.
When $\rank A | \mb{b} > \rank A$ there are no solutions.
If, on the other hand, $\rank A | \mb{b} = \rank A$, then we can choose a basis of $\Col A | \mb{b}$ from among the columns of $A$.
Using those columns as leading one columns when reducing $\Col A | \mb{b}$, we can reduce $\Col A | \mb{b}$ to reduced row echelon form without swapping the final column.
Essentially, we exclude $\mb{b}$ from the basis $B'$ in the proof of Theorem \ref{t_gauss}.
All column swaps left the final column $\mb{b}$ of $A | \mb{b}$ untouched. 
Thus, we can simply read off our solution space from the reduced form of $A | \mb{b}$ as usual.
If we do not know the rank of $A | \mb{b}$, then there is nothing we can say about the system $A\mb{x}=\mb{b}$.
The only potentially interesting case remaining is when $\rank A | \mb{b}$ known, but $\rank A$ unknown.
However, the simple system $\rho x = 1$ in $\mathbb{R}^1$ shows that also here there is not much to say.
Therefore, in practical real number arithmetic, our constructive version of Gauss-Jordan elimination is of no use.
We need to know the rank of our matrix before we can apply Theorem \ref{t_gauss}.
Yet, Gauss-Jordan elimination is one of the ways we usually calculate the rank of a matrix.
The following proposition confirms our suspicions.

\begin{proposition}
Let $A$ be a matrix that can be transformed into reduced row echelon form using elementary row operations and column swapping.
Then $A$ has a rank.
\end{proposition}
\begin{proof}
A matrix in reduced row echelon form is clearly ranked.
Its rank is the number of columns that contain a leading one.
We can transform the reduced row echelon matrix back into $A$ by sequentially applying the inverses of the operations used to obtain the reduced form.
By Proposition \ref{p_elem_col_operations_preserve_col_span} and \ref{p_elem_row_operations_preserve_col_rank} the rank of our matrix is not affected by these inverse operations.
\end{proof}

This leaves us with a somewhat unsatisfactory feeling.
Computing the rank of a matrix is generally undecidable.

\begin{proposition}
Let $T$ be the statement:
`Every matrix has a rank.'.
Then $T$ is equivalent to LPO.
\end{proposition}
\begin{proof}
Suppose that every matrix has a rank.
Let $\rho \in \mathbb{R}$ and consider the matrix $A = \lvert \rho \rvert$.
Clearly, knowing $\rank A = 0 \lor \rank A = 1$ is equivalent to $\rho = 0 \lor \rho \neq 0$.
Thus we can decide real number ordering and this is equivalent to LPO.
Next, assume LPO, then again we can decide $x = y \lor x \neq y$ for all real numbers $x,y$.
As seen before, the usual process of Gauss-Jordan elimination works.
We can reduce $A$ to reduced row echelon form and read its rank from the number of columns with a leading one.
\end{proof}

We will now investigate ways to approximate the rank of a matrix in some sense.
The notion of an $\varepsilon-rank$, denoted with $\rank_\varepsilon A$, captures that there are at least $\rank_\varepsilon A$ independent columns among the columns of $A \in \mathbb{R}^{n \times m}$.
This mirrors the findings of \cite{CtDoLS}, in particular Proposition 1.
If $A$ is ranked, then $\rank_\varepsilon A$ is obviously a lower bound on $\rank A$.
However, $\rank_\varepsilon A$ also exists for matrices $A$ that are not necessarily ranked.
We obtain an $\varepsilon$-rank of $A$ by computing a $\varepsilon$-\textit{row echelon form} of $A$ (Definition \ref{d_epsilon_row}).

As stated, $\rank_\varepsilon A = k$ means that the column space of $A \in \mathbb{R}^{m \times n}$ has an `approximate basis' of $k$ columns.
These columns are known to be independent.
The remaining $\min(n,m) - k$ columns of $A$, of which the independence is yet unknown, are negligible in some sense.
They contribute little to solutions of the system $A \mathbf{x} = \mathbf{0}$ (Proposition \ref{p_epsilon_kernel}).
This may seem a fairly useless result, as determining $\mathbf{b} = \mathbf{0}$ for the system $A \mathbf{x} = \mathbf{b}$ is undecidable.
However, systems of the form $A \mathbf{x} = \mathbf{0}$ naturally occur when finding eigenvectors.

\begin{definition} \label{d_epsilon_row}
Let $\varepsilon > 0$, then the $n \times m$ matrix $A$ is in $\bm{\varepsilon}$\textbf{-row echelon form} if:
\begin{enumerate}
\item the first $k$ columns of $A$ are in triangular form with diagonal entries apart from $0$,
\item the matrix $E$ consisting of the lower $n - k$ rows and columns satisfies $\norm{E}_{\max} < \varepsilon$.
\end{enumerate}
Such a matrix for instance looks like this:
\begin{equation*}
\begin{vmatrix}
+ & . & . & . & . & .\\
0 & + & . & . & . & .\\
0 & 0 & + & . & . & .\\
0 & 0 & 0 & * & * & *\\
0 & 0 & 0 & * & * & *\\
\end{vmatrix}.
\end{equation*}
Here, $+$ stands for an entry apart from $0$, $*$ is an entry whose absolute value does not exceed $\varepsilon$ and $.$ is an arbitrary entry.
We will also call the first $k$ columns that are in row echelon form the \textbf{leading} columns of the $\varepsilon$-row echelon form.
The matrix $E$ is called the \textbf{residual matrix}.
\end{definition}

As in the ranked case, we can reduce any matrix $A$ to a $\varepsilon$-row echelon form for any $\varepsilon > 0$ using elementary row and column operations.
In this setting, however, it is useful to only consider those elementary operations of `small norm'.
This ensures a $\varepsilon$-row echelon form of $A$ is `close' to $A$ in some sense, which we when we investigate approximate kernels of $A$ in Proposition \ref{p_epsilon_kernel}.
We call these operations \textit{stable} elementary operations.

\begin{definition}
We call the elementary operations of row swapping, row addition and column swapping the \textbf{stable} (elementary) operations.
If $A_\varepsilon$ is an $\varepsilon$-row echelon form that is obtained from the matrix $A$ by applying stable operations, then we also call $A_\varepsilon$ an $\varepsilon$-row echelon form \textbf{of} $A$.
\end{definition}

Using a limited number of stable elementary operations, we can reduce any matrix to a $\varepsilon$-row echelon form.
This is the content of Proposition \ref{p_epsilon_reduction}.
Moreover, the resulting $\varepsilon$-row echelon form is close to the original matrix in the \textit{maximum} norm.
Thus we need a small lemma on the maximum norm of the product of two matrices before we can derive Proposition \ref{p_epsilon_reduction}.

\begin{definition}
Let $A \in \mathbb{R}^{n \times m}$, then the \textbf{maximum norm} $\norm{A}_{\max}$ of $A$ is defined as $\max_{i,j} \lvert A_{i,j} \rvert$.
\end{definition}

\begin{lemma} \label{l_max_norm_product}
Let $A \in \mathbb{R}^{n \times m}$ and $B \in \mathbb{R}^{m \times k}$, then $\norm{AB}_{\max} \leq m \norm{A}_{\max}\norm{B}_{\max}$.  
\begin{proof}
Exactly as in classical mathematics, for all column indices $1 \leq i \leq k$ we have: \\
$
    \norm{[AB]_i}_{\max} = \norm{\sum_{j = 1}^m [A]_j [B]_{ji}} \leq \sum_{j = 1}^m \norm{[A]_j [B]_{ji}}_{\max} \leq \sum_{j = 1}^m \norm{A}_{\max}\norm{B}_{\max} = m \norm{A}_{\max}\norm{B}_{\max}
$.
\end{proof}
\end{lemma}

\begin{proposition} \label{p_epsilon_reduction}
Let $A$ be an $n \times m$ matrix and let $\delta, \varepsilon > 0$, then an $\varepsilon$-row echelon form $A_\varepsilon$ of $A$ exists.
Moreover we can write $A_\varepsilon = R_1 S_1\ ... \ R_k S_k A C_1 \ ... C_k$ with $k \leq \min(n,m)$.
Here for each $i$, $R_i$ is an identity matrix where two rows are swapped and $C_i$ is an identity matrix where two columns are swapped.
Finally, for each $i$, $S_i$ is a row addition matrix satisfying $\norm{S_i}_{\max} < 1 + \delta$.
\end{proposition}
\begin{proof}
We use induction on the number of columns, so first take $m = 1$.
If $\lvert a_{i,1} \rvert < \varepsilon$ for all $i$ we are done.
When $\lvert a_{i,1} \rvert > 0$ for some $i$, then we know $\max_{i} \lvert a_{i,1} \rvert > 0$.
For each $j$, decide $\lvert a_{j,1} \rvert > (1 + \delta)^{-1} \max_{i} \lvert a_{i,1} \rvert  \ \lor \ \lvert a_{j,1} \rvert <  \max_{i} \lvert a_{i,1} \rvert$. 
As it is contradictory that $\lvert a_{j,1} \rvert < \max_{i} \lvert a_{i,1} \rvert$ for each $j$, we find a $j$ with $\lvert a_{j,1} \rvert > (1 + \delta)^{-1} \max_{i} \lvert a_{i,1}\rvert $.
Move this $j$'th entry of the matrix $A$ to the first row by applying row swaps.
Subsequently we subtract suitable multiples of the first row from the other rows.
The  factor used for row $k$ is $\mu_k=a_{k,1} / a_{1,1}$, and hence $\lvert \mu_k \rvert < \max_i \lvert a_{i,1} \rvert / \big((1 + \delta)^{-1}  \max_i \lvert a_{i,1} \rvert \big) = 1 + \delta$.
This proves the base case $m = 1$.
For $m > 1$, either $\lvert a_{i,j} \rvert < \varepsilon$ for all $i,j$ and we are done, or we can find $i,j$ such that $\rvert a_{i,j} \lvert > 0$.
Swap column $j$ and $1$ and convert the first column to $(a_{i,j} \ 0 \ ... \ 0)^T$ as in the base case.
The sub-matrix of our transformed matrix obtained by omitting the first row and column is reducible to a $\varepsilon$-row echelon form by the induction hypothesis.
The used row operations leave the first column of our transformed matrix unaffected and we are done.
Note that the induction step uses at most one column swap, one row swap and $n - 1$ row additions.
Here we also count a swap matrix if we are in fact not swapping at all.
Stated otherwise, an identity matrix can also count as a row swap or column swap matrix.
The $n - 1$ row additions are all part of a single matrix $S_i$.
Since $k$ also equals the number of diagonal entries apart from $0$ in the $\varepsilon$-row echelon form we produced,
it is immediately clear that $k \leq \min(n,m)$.
\end{proof}

\begin{corollary} \label{c_epsilon_form_norm}
Let $A$ be an $n \times m$ matrix and let $\delta, \varepsilon > 0$ and define $k = \min(n,m)$.
Then there is an $\varepsilon$-row echelon form $A_\varepsilon$ of $A$ that satisfies $\norm{A_\varepsilon}_{\max} \leq \norm{A}_{\max} n^{2 k}m^{k}(1 + \delta)^{k}$.
\begin{proof}
This follows directly from the previous proposition and Lemma \ref{l_max_norm_product}. 
\end{proof}
\end{corollary}

Proposition \ref{p_epsilon_reduction} is essentially the partial pivoting strategy of numerical linear algebra.
There it is necessary to obtain numerical stability of Gauss-Jordan elimination.
Since finding an $j$ such that $\lvert \mathbf{a}_j \rvert = \max_i \lvert \mathbf{a}_i \rvert$ is undecidable, we need the relaxation factor $(1 + \delta)^{-1}$.
Moreover norm estimates on matrices resulting from Gaussian elimination, as in Corollary \ref{c_epsilon_form_norm}, are also ubiquitous in classical, numerical linear algebra.
See for instance the classic paper \cite{EAoDMoMI}.
Having shown existence of $\varepsilon$-row echelon forms, we can define the $\varepsilon$-rank of a matrix.

\begin{definition} \label{d_epsilon_rank}
Let $A$ be a matrix and $\varepsilon > 0$.
The number of leading columns of a $\varepsilon$-row echelon form of $A$ is called an \textbf{$\varepsilon$-rank} of $A$.
We will also denote this as $\rank_\varepsilon A$.
\end{definition}

Note how Definition \ref{d_epsilon_row} does not define $\rank_\varepsilon A$ as \textit{the} $\varepsilon$-rank of the matrix $A$.
That would be incorrect. 
The operation assigning to each matrix $A$ and $\varepsilon > 0$ the natural number $\rank_\varepsilon A$, which is implicit in Proposition \ref{p_epsilon_reduction}, is not a \textit{function}.
It is not an \textit{extensional} operation, if $A = B$ and $\varepsilon = \delta$, this does not mean $\rank_\varepsilon A = \rank_\delta B$.
This is because we crucially rely on deciding propositions of the form $\lvert x \rvert > 0 \lor \lvert x \rvert < \varepsilon$.
We may view this as another operation mapping $(x, \varepsilon)$ to $\{0, 1\}$.
Here $0$ means $\lvert x \rvert > 0$ and only $\lvert x \rvert > 0$ is established and $1$ means $\lvert x \rvert < \varepsilon$ is known.
This operation is also inherently not extensional.
The answer to the question `$\lvert x \rvert > 0$ or $\lvert x \rvert < \varepsilon$' depends on exactly what form we know $x$ and $\varepsilon$ and how we establish which leg of the disjunction we choose.

Take $x = 1$ and $\varepsilon = 2$.
Note that with this \textit{intensional} information on $x$ and $\varepsilon$, it is immediate that $0 < x < \varepsilon$.
But without this knowledge, we rely on the actual construction of $x$ and $\varepsilon$ as real numbers from Definition \ref{d_real}.
For instance, take $x_n = [0, 1]$ if $n < 1000$ and $[1,1]$ otherwise, $\varepsilon_n = [0, 2]$ for all $n < 1000$, $\varepsilon_{1000} = [1, 2]$ and $\varepsilon_n = [2,2]$ if $n > 1001$.
If we then decide $\lvert x \rvert > 0 \lor \lvert x \rvert < \varepsilon$ by using the 1000'th element of $x$ and $\varepsilon$, we only know $\lvert x \rvert > 0$.

Note that we cannot circumvent this issue by e.g. constructing the $\varepsilon$-row echelon form in another way.
Fixing $\varepsilon > 0$, the mapping of $A$ to $\rank_\varepsilon A$ is an operation from $\mathbb{R}^{n \times m}$ to $\mathbb{N}$.
We wish this to be a \textit{total} mapping, so defined for all matrices in $\mathbb{R}^{n \times m}$.
If it were extensional as well, then this mapping is by definition a function from $\mathbb{R}^{n \times m}$ to $\mathbb{N}$.
We can not constructively prove the existence of such functions, see for instance page 120 of \cite{VoCM}.

We should really view $\rank_\varepsilon A$ as the operation of $\varepsilon$-reduction applied to $A$ and subsequently counting the number of leading columns.
A notation such as $\rank[\varepsilon, A]$ would be more appropriate, the square brackets signifying that this is not a function.
Definition \ref{d_epsilon_rank} is thus a slight abuse of notation.
The crux is that the natural number $\rank_\varepsilon A$ can be shown to have certain useful properties.
Examples include the next Lemma and Proposition \ref{p_det_and_rank}.
These properties of $\rank_\varepsilon A$ hold even though its value changes depending on how we check $\lvert x \rvert > 0 \lor \lvert x \rvert < \varepsilon$ and how $A$ and $\varepsilon$ are constructed.

\begin{lemma}
Let $A$ be a ranked matrix and $\varepsilon > 0$, then $\rank_\varepsilon A \leq \rank A$.
\begin{proof}
An $\varepsilon$-row echelon form $A_\varepsilon$ of $A$ contains $\rank_\varepsilon A$ independent columns.
As $A_\varepsilon$ is obtained from $A$ by stable elementary operations, the matrix $A$ also contains $\rank_\varepsilon A$ independent columns.
Column swapping does not affect independence of columns and the elementary row operations do not affect rank by Proposition \ref{p_elem_row_operations_preserve_col_rank}.
\end{proof}
\end{lemma}

\section{Eigenvalues and eigenvectors} \label{sec_eigen}

Defining the \textit{determinant} $\det A$ of a square matrix $A$ presents no constructive difficulties.
Most familiar properties still hold.
We refer the reader to section 4.11 of \cite{CiM2}.
To prove some additional properties of the determinant, or more precisely of singular matrices with $\det A = 0$, we need the bounds provided by Proposition \ref{p_epsilon_reduction} on the reduction matrices.
In particular, we will work towards Theorem \ref{t_det_zero}.
This shows the equivalence between $\det A = 0$ and the existence of approximate kernel vectors up to any precision.
Stated otherwise, it provides a method to calculate approximate eigenvectors of $A$ with eigenvalue $\lambda$ from knowing $\det A - \lambda I = 0$. 
Propositions \ref{p_det_and_rank} and \ref{p_epsilon_kernel} are the main technical results needed to prove Theorem \ref{t_det_zero}.
In turn, these propositions require some elementary lemmata on norm approximations we derive below.

\begin{definition}
If $A \in \mathbb{R}^{n \times m}$, then the $L^2$-norm $\norm{A}$ of $A$ is defined as $\sup_{\{ \mathbf{x} \ | \ \norm{\mathbf{x}} = 1 \}} \norm{A\mathbf{x}}$, the norm induced by the $L^2$-vector norm.
\end{definition}
That the matrix norm $\norm{A}$ is actually computable and a norm is shown in Corollary 4.1.8 of \cite{ToCA}.

\begin{lemma}[Cauchy-Schwarz] \label{l_cauchy_schwarz}
Let $\mathbf{x}, \mathbf{y}$ be vectors in $\mathbb{R}^n$, then $\langle \mathbf{x}, \mathbf{y} \rangle^2 \leq \langle \mathbf{x}, \mathbf{x} \rangle\langle \mathbf{y}, \mathbf{y} \rangle$.
\begin{proof}
Suppose first we could decide $\mathbf{y} = \mathbf{0} \lor \mathbf{y} \neq \mathbf{0}$.
Then we would reason as in classical mathematics.
If $\mathbf{y} = \mathbf{0}$, then both $\langle  \mathbf{x}, \mathbf{y} \rangle$ and $\langle \mathbf{y}, \mathbf{y} \rangle$ are $0$ and we are done.
When $\mathbf{y} \neq \mathbf{0}$, then we look at the vector $\mathbf{z} = \mathbf{x} - \langle \mathbf{x}, \mathbf{y} \rangle/\langle \mathbf{y}, \mathbf{y} \rangle \mathbf{y}$.
This is the orthogonal complement of $\mathbf{x}$ projected on $\mathbf{y}$.
The Cauchy-Schwarz inequality then easily follows from $\norm{\mathbf{z}} \geq 0$.

Of course, we cannot decide $\mathbf{y} = \mathbf{0} \lor \mathbf{y} \neq \mathbf{0}$.
However, we must prove a negative statement, $\neg (\langle \mathbf{x}, \mathbf{y} \rangle^2 > \langle \mathbf{x}, \mathbf{x} \rangle\langle \mathbf{y}, \mathbf{y} \rangle)$.
So suppose $\langle \mathbf{x}, \mathbf{y} \rangle^2 > \langle \mathbf{x}, \mathbf{x} \rangle\langle \mathbf{y}, \mathbf{y} \rangle$, we will arrive at contradiction.
In that case $\neg (\mathbf{y} = \mathbf{0} \lor \mathbf{y} \neq \mathbf{0})$ follows from the reasoning above.
But we know $\neg (\mathbf{y} = \mathbf{0} \lor \mathbf{y} \neq \mathbf{0})$ implies $\neg (\mathbf{y} = \mathbf{0}) \land \neg (\mathbf{y} \neq \mathbf{0})$, so $\neg (\mathbf{y} = \mathbf{0}) \land \mathbf{y} = \mathbf{0}$.
This contradiction shows $\neg (\langle \mathbf{x}, \mathbf{y} \rangle^2 > \langle \mathbf{x}, \mathbf{x} \rangle\langle \mathbf{y}, \mathbf{y} \rangle)$.
\end{proof}
\end{lemma}

\begin{corollary} \label{c_cauchy_schwarz}
Let $\mathbf{x}$ be a vector in $\mathbb{R}^n$, then $\sum_{i = 1}^n \lvert \mathbf{x}_i \rvert \leq \sqrt{n} \norm{\mathbf{x}}$.
\begin{proof}
Apply Cauchy-Schwarz to $\mathbf{1}$ and the vector of the absolute values of the entries of $\mathbf{x}$.
\end{proof}
\end{corollary}

The proof of Lemma \ref{l_cauchy_schwarz} uses a technique which Berger and Svindland call the `*-rule' \cite{MUCaP}.
This `logical trick' is much older than their article however.
The first reference we could find in the literature is the above proof of the Cauchy-Schwarz inequality in Kumar's PhD thesis \cite{HSiI}.
In abstract form the reasoning is as follows.
We wish to derive $\neg P$ for some statement $P$.
We consider statements $A_1$, $A_2$ ... $A_n$ and first prove $(A_1 \lor A_2 \lor \ ... \lor A_n) \rightarrow \neg P$.
Subsequently we see that $\neg \neg(A_1 \lor A_2 \lor \ ... \lor A_n)$ leads to $\neg P$.
The crucial trick part of the trick is that $(A_1 \lor A_2 \lor \ ... \lor A_n) \rightarrow \neg P$ is usually trivial and follows the classical proof.
While $A_1 \lor A_2 \lor \ ... \lor A_n$ is generally undecidable, $\neg \neg(A_1 \lor A_2 \lor \ ... \lor A_n)$ does hold constructively and is sufficient to prove $\neg P$.
Often the statements $A_1$, $A_2$ ... $A_n$ are mutually exclusive statements that classically exhaust all options.
An example is $x<0$, $x=0$ and $x>0$ for any real number $x$.

\begin{lemma} \label{l_max_norm_vector}
Let $\mathbf{x}$ be a vector in $\mathbb{R}^n$, then $ \norm{\mathbf{x}}_{\max} \leq \norm{\mathbf{x}} \leq \sqrt{n} \norm{\mathbf{x}}_{\max}$.
\begin{proof}
Exactly as in classical mathematics, $\norm{x}^2 = \sum_{i = 1}^n \mathbf{x}_i^2 \leq \sum_{i = 1}^n \norm{\mathbf{x}}_{\max}^2 = n\norm{\mathbf{x}}_{\max}^2$.
\end{proof}
\end{lemma}

\begin{lemma} \label{l_max_norm}
Let $A$ be a matrix in $\mathbb{R}^{n \times n}$, then $\norm{A} \leq n \norm{A}_{\max}$.
\begin{proof}
Exactly as in classical mathematics.
Consider the $i$'th element of $A\mathbf{x}$, so $[A\mathbf{x}]_i = \sum_{j = 1}^n A_{ij}\mathbf{x}_j$.
Then we see $\lvert [A\mathbf{x}]_i \rvert \leq \norm{A}_{\max} \sum_{j = 1}^n \lvert \mathbf{x}_j \rvert$.
From Corollary \ref{c_cauchy_schwarz} we get $\norm{A\mathbf{x}}_{\max} \leq \sqrt{n} \norm{A}_{\max} \norm{\mathbf{x}}$ for all $\mathbf{x}$.
Applying the previous lemma, we get $\norm{A\mathbf{x}} \leq n \norm{A}_{\max} \norm{\mathbf{x}}$.
Combining this inequality with the definition of the matrix $L^2$-norm, we are done.
\end{proof}
\end{lemma}

\begin{proposition} \label{p_det_and_rank}
Let $A$ be an $n \times n$ matrix, then $\det A = 0$ if and only if $\rank_\varepsilon A$ is less than $n$ for all $\varepsilon > 0$.
\begin{proof}
Suppose that $\det A = 0$.
Any $\varepsilon$-row echelon form $A_\varepsilon$ of $A$ is obtained from $A$ using stable operations.
These do not change the determinant, so we have $\det A_\varepsilon = 0$.
If the number of leading columns of $A_\varepsilon$ were $n$, then $A_\varepsilon$ is a triangular matrix with diagonal entries apart from $0$.
Obviously that would mean $\det A_\varepsilon \neq 0$.
This contradiction implies the number of leading columns is less than $n$ and hence $\rank_\varepsilon A < n$.

Next assume the $\varepsilon$-rank of $A$ is less than $n$ for all $\varepsilon > 0$.
We will show that $\lvert \det A \rvert < \varepsilon'$ for any $\varepsilon' > 0$.
As a $\varepsilon$-form $A_\varepsilon$  of $A$ is obtained from $A$ through stable operations, we have $\lvert \det A \rvert = \lvert \det A_\varepsilon \rvert$ for any $\varepsilon$.
In particular this holds for those $A_\varepsilon$ constructed using Corollary \ref{c_epsilon_form_norm}.
The determinant of such $A_\varepsilon$ will be the product of the determinant of the leading columns and the residual matrix.
This holds because the leading columns are a triangular form.
In particular the determinant of these leading columns is the product of the diagonal entries.
Taking $\delta = 1$ in Corollary \ref{c_epsilon_form_norm}, the absolute values of these entries are bounded above by $n^{3n}2^n\norm{A}_{\max}$.
Moreover our assumption implies there are at most $n - 1$  diagonal entries in the leading columns of $A_\varepsilon$.
Defining $\rho(A,n) = n^{3n}2^n\norm{A}_{\max}$, the absolute value of the product of these entries is thus at most $\rho(A,n)^{n - 1}$.
Next we consider the residual matrix $E$ of $A_\varepsilon$.
We know this matrix consists of at least one column because $\rank_\varepsilon A < n$.
As all entries of the $(n - \rank_\varepsilon A) \times (n - \rank_\varepsilon A)$ matrix $E$ do not exceed $\varepsilon$ in absolute value, its determinant is bounded by $(n - \rank_\varepsilon A)! \lvert \varepsilon \rvert^{n - \rank_\varepsilon A}$.
For any $\varepsilon < 1/n!$ this is bounded above by $\varepsilon$.
Now pick $\varepsilon = \min \big(1/n!, \varepsilon' / \rho(A,n)^{n - 1} \big)$, then $\lvert \det A_\varepsilon \rvert < \varepsilon'$.
Thus $\lvert \det A \rvert = \lvert \det A_\varepsilon \rvert < \varepsilon'$ and as $\varepsilon' > 0$ is arbitrary, this shows $\det A = 0$.
\end{proof}
\end{proposition}

\begin{proposition} \label{p_epsilon_kernel}
Let $A$ be an $\mathbb{R}^{n \times n}$ matrix with $\det A = 0$.
Then for every $\varepsilon > 0$ there exist $l = n - \rank_\varepsilon A > 0$ vectors $\mathbf{x} \in \mathbb{R}^n$ with $\norm{\mathbf{x}} = 1$ such that $\norm{A\mathbf{x}} < \varepsilon$.
Moreover for each $\beta \in \mathbb{R}^l$ we have $\norm{\sum_{i = 1}^l \beta_i \mathbf{x}_i} \geq M_A \sum_{i = 1}^l \lvert \beta_i \rvert$ for some $M_A > 0$.
\begin{proof}
Let $\varepsilon' = \varepsilon / \big(\sqrt{n} (2n)^{2n}\big)$ and construct an $\varepsilon'$-row reduced form $B$ of $A$.
By Proposition \ref{p_det_and_rank} the matrix $B$ has fewer than $n$ leading columns and therefore looks as follows:

\begin{equation*}
\begin{vmatrix}
T & M \\
0 & E \\
\end{vmatrix}.
\end{equation*}

Defining $k = \rank_\varepsilon A$, then $T$ is a $k \times k$ triangular form, $M$ is an arbitrary $k \times (n - k)$ matrix and $E$ is the $(n - k) \times (n - k)$ residual matrix. 
From the previous proposition we know $k < n$.
Obviously $T$ has rank $k$ and is therefore invertible.
For each column $i \in \{k + 1 \ ... \ n \}$, define the vector $\mb{y}_i \in \mathbb{R}^n$ as the vector whose first $k$ values are $T^{-1} (-M_i)$ and last $n - k$ values are all $0$ except the $i$'th entry, which is $1$.
The vector $B \mathbf{y}_i$ will then consist of $k$ zeroes, followed by $n - k$ entries whose absolute values are all less than $\varepsilon'$.
That means $\norm{B \mathbf{y}_i}^2 \leq n\varepsilon'^2$ by Lemma \ref{l_max_norm_vector}.
Because the $j$'th entry of $\mathbf{y}_i$ is $1$, we know $\norm{y_i} \geq 1$.
So defining $\mb{z}_i = \mb{y}_i / \norm{\mb{y}_i}$ we get $\norm{B\mb{z}_i} < \norm{B \mathbf{y}_i} \leq \sqrt{n}\varepsilon' \leq \varepsilon$.
By Proposition \ref{p_epsilon_reduction} we can write $B = R_1 \ ... \ R_{2n} A C_1 \ ... \ C_n$, where the $R_i$ are stable row operations satisfying $\norm{R_i}_{\max} < 2$ for all $i$ and the $C_i$ are column swaps.
Now consider the vector $\mb{x}_i = C_1 \ ... \ C_n\mb{z}_i$.
This vector also has norm $1$ and we get $A\mb{x} =  R_{2n}^{-1} \ ... \ R_1^{-1} B C_n^{-1} \ ... C_1^{-1}\mb{x}_i = R_{2n}^{-1} \ ... \ R_1^{-1} B \mb{z}_i$.
Since $\norm{R_i}_{\max} < 2$, we know $\norm{R_i} < 2n$ by Lemma \ref{l_max_norm}.
We then get $\norm{A\mb{x}_i} \leq \norm{R_{2n}^{-1}} \ ... \ \norm{R_1^{-1}} \norm{B \mb{z}_i} \leq \norm{R_{2n}^{-1}} \ ... \ \norm{R_1^{-1}} \sqrt{n(\varepsilon')^2} = (2n)^{2n}\sqrt{n}\varepsilon' = \varepsilon$ and we found $n - k$ unit norm vectors whose product with $A$ has length less than $\varepsilon$.

For each $\varepsilon$-eigenvector $\mathbf{x}$, we can find an index $j$ such that $[\mathbf{x}]_j \neq 0$, but $[\mathbf{v}]_j = 0$ for all other $\varepsilon$-eigenvectors $\mathbf{v}$.
Thus, we have shown the second claim of this proposition if we can find a lower bound on the absolute values of these entries that only depends on $A$.
The result then immediately follows from Corollary \ref{c_cauchy_schwarz}.
Each vector $\mathbf{x}_i$ is the product of $C_1 \ ... \ C_n$ and a vector $\mathbf{z}_i$. 
Those vectors $\mathbf{z}_i$ are normalized versions of vectors $\mathbf{y}_i$, which have the property that for each $i$, we can find a $j$ such that $[\mathbf{y}_i]_j = 1$, but $[\mathbf{y}_k]_j = 0$ for all $k \neq j$.
So we only have to find an upper bound for the norm of $\mathbf{y}_i$.
By Lemma \ref{l_max_norm_vector} and \ref{l_max_norm_product} we have $\norm{\mathbf{y}_i} \leq \sqrt{n}\max(1, \norm{T^{-1}M_i}_{\max}) \leq \sqrt{n}\max(1, n\norm{T^{-1}}_{\max} \norm{M}_{\max})$.
We can bound $\norm{M}_{\max}$ by $\norm{A_\varepsilon}_{\max}$ and we have a bound for $\norm{A_\varepsilon}_{\max}$ in Corollary \ref{c_epsilon_form_norm}.
That leaves $\norm{T^{-1}}_{\max}$.
As $T$ is matrix of full rank, we can find $T^{-1}$ with standard Gaussian elimination.
No columns swaps are needed and we can use the pivoting strategy from Proposition \ref{p_epsilon_reduction}.
We multiply $T$ with row swap and row addition matrices until it is reduced to $I$.
The product of these row operation matrices is $T^{-1}$.
We then find a similar norm estimate as in Corollary \ref{c_epsilon_form_norm}, $\norm{T^{-1}}_{\max} \leq n^{2n}(1 + \delta)^n \norm{T}_{\max}$.
Since obviously $\norm{T}_{\max} \leq \norm{A_\varepsilon}_{\max}$, we are done.
\end{proof}
\end{proposition}

Our main contribution Theorem \ref{t_det_zero} is now a simple consequence Proposition \ref{p_epsilon_kernel} and a standard continuity argument.

\begin{theorem} \label{t_det_zero}
Let $A$ be an $\mathbb{R}^{n \times n}$ matrix.
Then $\det A = 0$ if and only if for every $\varepsilon > 0$ there exists an $\mathbf{x} \in \mathbb{R}^n$ with $\norm{\mathbf{x}} = 1$ such that $\norm{A\mathbf{x}} < \varepsilon$. 
\begin{proof}
The `only if' part is the content of the previous proposition.
For the other implication, assume that for every $\varepsilon > 0$ there exists an $\mathbf{x} \in \mathbb{R}^n$ with $\norm{\mathbf{x}} = 1$ such that $\norm{A\mathbf{x}} < \varepsilon$.
Furthermore, suppose $\det A \neq 0$.
Then $A$ is invertible by Theorem 4.14 of \cite{CiM2}. 
This essentially follows Cramer's rule, whose proof is constructive once we assume $\det A \neq 0$.
Let $\varepsilon_k \rightarrow 0$ and construct a sequence $\mb{x}_k \in \mathbb{R}^n$ such that $\norm{\mb{x}_k} = 1$ and $\norm{A \mb{x}_k} < \varepsilon_k$ for all $k$.
Define $\mb{y}_k = A\mb{x}_k$, then clearly $\mb{y}_k \rightarrow \mb{0}$ as $k \rightarrow \infty$.
Because the mapping $\mb{x} \mapsto A^{-1}\mb{x}$ is continuous, this implies $A^{-1}\mb{y}_k$ converges to $A^{-1}\mb{0} = \mb{0}$.
But $A^{-1}\mb{y}_k = \mb{x}_k$ and $\norm{\mb{x}_k} = 1$ for all $k$.
This contradiction forces us to conclude $\neg \det A \neq 0$ and hence $\det A = 0$.
\end{proof}
\end{theorem}
If $A$ is an $n \times n$ matrix and $\mathbf{x} \neq \mathbf{0}$ an $n \times 1$ vector such that $A\mathbf{x} = \lambda \mathbf{x}$ for some $\lambda \in \mathbb{R}$, then we say $\mathbf{x}$ is an \textit{eigenvector} of $A$ and $\lambda$ is an \textit{eigenvalue}.
We will see that these concepts are not well-behaved in a constructive setting.
Although roots of the characteristic polynomial $\lambda \mapsto \det A - \lambda I$ can be calculated, finding the corresponding eigenvectors is undecidable.
We formally prove this undecidability in Proposition \ref{p_eigenvectors_uncomputable}.
It turns out that even for ranked, square matrices, eigenvectors are not computable.
\clearpage
For this reason, Osinenko, Devadze, and Streif consider $\varepsilon$-eigenvalues and $\varepsilon$-eigenvectors in \cite{CAoEPiCuNU}.
They first approximate $A$ with a rational matrix $A^*$. 
Eigenvalues $\lambda^*$ of $A^*$ are algebraic numbers, whose ordering and equality are decidable.
Subsequently we can find the eigenvectors of $A^*$ essentially as in classical mathematics.
Using standard Gaussian elimination for instance, we only encounter algebraic numbers as matrix entries.
Exact eigenvectors of $A^*$ then yield $\varepsilon$-eigenvectors of $A$. 

Their method is elegant but introduces significant computational overhead, as one must track the corresponding polynomials of all entries under every addition and multiplication.
Our approach treats matrix entries as abstract real numbers (black boxes providing rational approximations), avoiding such overhead. 
The method is more `linear algebraic' in nature.
In this section we present this alternative construction of $\varepsilon$-eigenvectors based on Theorem \ref{t_det_zero}.

Note that $\varepsilon$-eigenvectors were first introduced in an infinite dimensional setting by Bridges and Ishihara in \cite{SoSOiCA}.
However, here only selfadjoint operators are considered.
But both our work and that of Osinenko, Devadze, and Streif (\cite{CAoEPiCuNU}) considers general, real matrices, not just symmetric matrices.

\begin{definition}
Let $A$ be an $n \times n$ matrix and $\varepsilon > 0$.
A tuple $(\lambda, \mathbf{x}) \in \mathbb{R} \times \mathbb{R}^n$ is an \textbf{$\varepsilon$-eigenpair} if:
\begin{enumerate}
\item $\lambda$ is an \textbf{eigenvalue}, so $\det A - \lambda I = 0$, 
\item $\mathbf{x}$ is a normalized \textbf{$\varepsilon$-eigenvector} with eigenvalue $\lambda$, so $\norm{A\mathbf{x} - \lambda\mathbf{x}} < \varepsilon$ and $\norm{\mathbf{x}} = 1$. 
\end{enumerate}
\end{definition}

Calculating eigenvalues turns out be easy constructively.
This is a direct consequence of the Fundamental Theorem of Algebra (Theorem \ref{t_fta}).

\begin{definition} \label{d_poly}
A \textbf{polynomial of degree} $n$ is a function $p:\mathbb{C} \rightarrow \mathbb{C}$ that maps $x \in \mathbb{C}$ to $\sum_{i = 0}^n a_ix^i$ with $(a_i)_{0 \leq i \leq n} \in \mathbb{R}$ and $a_n \neq 0$.
When $n$ is clear from the context, we will also just call $p$ a \textbf{polynomial}.
\end{definition}

\begin{theorem}[Fundamental Theorem of Algebra]\label{t_fta}
Let $p$ be a polynomial of degree $n$, then there exist complex numbers $z_i$ with $1 \leq i \leq n$ such that $p(z)=a_n\prod_{i =1}^n(z - z_i)$ for all $z \in \mathbb{C}$.
\begin{proof}
See Theorem 5.10 of \cite{CA}.
\end{proof}
\end{theorem}

Thus it is immediately clear that for any matrix $A \in \mathbb{R}^{n \times n}$ we can compute $n$ eigenvalues $\lambda_1,...,\lambda_n$.
The tricky part, however, is the calculation of the corresponding eigenvectors.
Exact eigenvectors are fundamentally uncomputable, as Proposition \ref{p_eigenvectors_uncomputable} shows.
We use a construction mentioned on Math Overflow (https://mathoverflow.net/questions/369930/why-is-uncomputability-of-the-spectral-decomposition-not-a-problem?).
Theorem 1.7 of \cite{DoCMaaRoI} in fact strengthens Proposition \ref{p_eigenvectors_uncomputable} to an equivalence with a real number principle between LLPO and LPO.

\begin{proposition}\label{p_eigenvectors_uncomputable}
Let $T$ be the statement
`Let $A$ be a ranked, square matrix with at least two columns.
Then $A$ has an eigenvector'.
Then $T$ implies LLPO.
\begin{proof}
We first show that $T$ implies LLPO.
Let $\rho \in \mathbb{R}$ and define $\rho_1=\min(0,\rho)$ and $\rho_2=\max(0,\rho)$.
Consider the following matrix $A$:
\begin{equation*}
\begin{vmatrix}
1 + \rho_1 & \rho_2 \\
\rho_2 & 1 - \rho_1
\end{vmatrix}.
\end{equation*}
Its characteristic polynomial is $\lambda^2 - 2\lambda + 1 - \rho_1^2 - \rho_2^2$.
This has complex roots $1 \pm \sqrt{\rho_1^2 +\rho_2^2}$.
Since $\rho_1^2 = \min(0, \rho^2) = \min(0, \rho^2)$ and $\rho_2^2 = \max(0, \rho)^2 = \max(0, \rho^2)$,
we get $\rho_1^2 +\rho_2^2 = \rho^2$.
Hence the roots of the characteristic polynomial are $1 \pm \lvert \rho \rvert$.
If $(a, \sqrt{1 - a^2})^T$ is a normalized eigenvector of $A$, then we have the following system of equations:
\begin{align*}
       a\rho_1 + \sqrt{1 - a^2} \rho_2 &= a \lvert \rho \rvert \\
       a\rho_2 - \sqrt{1 - a^2} \rho_1 &= \sqrt{1 - a^2} \lvert \rho \rvert.
\end{align*}
Decide $a > 0 \lor a < 1/2$ and first consider $a > 0$.
Suppose $\rho < 0$, then $\lvert \rho \rvert = -\rho$, $\rho_1 = \rho$ and $\rho_2 = 0$.
Our system of equations reduces to:
\begin{align*}
       a\rho &= -a \rho \\
      \sqrt{1 - a^2} \rho &= -\sqrt{1 - a^2} \rho.
\end{align*}
Because $a > 0$, the first equation is contradictory and we must conclude $\rho \geq 0$.
Next consider $a < 1/2$ and suppose $\rho > 0$.
Then we know $\lvert \rho \rvert = \rho$, $\rho_1 = 0$ and $\rho_2 = \rho$.
The system of equations for our eigenvector $(a, \sqrt{1 - a^2})^T$ becomes:
\begin{align*}
       \sqrt{1 - a^2} \rho &= a \rho \\
       a\rho &= \sqrt{1 - a^2} \rho. 
\end{align*}
Solving for $a$ yields $a = (1/2)\sqrt{2}$, which contradicts $a < 1/2$.
Therefore, we must now conclude $\rho \leq 0$.
Thus we are able to decide $\rho \leq 0 \lor \rho \geq 0$ for an arbitrary $\rho \in \mathbb{R}$, which is precisely LLPO.
\end{proof}
\end{proposition}

Like Osinenko, Devadze, and Streif in \cite{CAoEPiCuNU}, we thus turn our attention towards $\varepsilon$-eigenvectors.
These are trivially computable from Proposition \ref{p_epsilon_kernel}.

\begin{proposition}
Let $A$ be a square $n \times n$ matrix, $\lambda$ an eigenvalue of $A$ and $\varepsilon > 0$.
Then there exist $n - \rank_\varepsilon (A - \lambda I)$ independent $\varepsilon$-eigenvectors of $A$ with eigenvalue $\lambda$.
\begin{proof}
This follows immediately  from $\det A - \lambda I = 0$ and Proposition \ref{p_epsilon_kernel}.
\end{proof}
\end{proposition}

Proving that we can also obtain independent $\varepsilon$-eigenvectors from Proposition \ref{p_epsilon_kernel} for $\varepsilon$ small enough is more difficult.
The following lemmata and propositions work towards this result in the form of Theorem \ref{t_eigen}.

\begin{lemma} \label{l_beta}
Let $X = \{ \mathbf{x}_i \ | \ 1 \leq i \leq m\}$ be a set of vectors in $\mathbb{R}^n$.
Suppose every subset of $m - 1$ vectors of $X$ is independent.
If $\beta\in \mathbb{R}^m$ satisfies $\beta \neq \mathbf{0}$, then either $\sum_i \beta_i \mathbf{x}_i \neq \mathbf{0}$, or $\beta_i \neq 0$ for each $i$.
\begin{proof}
Since $\beta \neq \mathbf{0}$, there is a $k$ with $\beta_k \neq 0$.
Without loss of generality we take $k = 2$.
We first show $\sum_i \beta_i \mathbf{x}_i \neq \mathbf{0}$ or $\beta_l \neq 0$ for every $l \neq 2$.
Again without loss of generality, we only look at the case $l = 1$.
By our assumption on $X$, the linear combination $\sum_{i=2}^m \beta_i \mathbf{x}_i$ is apart from $\mathbf{0}$.
This holds because the term $\beta_2$ is part of the coefficients over which we sum.
Thus there is a $j$ such that $(\sum_{i=2}^m \beta_i \mathbf{x}_i)_j \neq 0$.
Decide $\lvert (\beta_1 \mathbf{x}_1)_j \rvert > 0 \lor \lvert (\beta_1 \mathbf{x}_1)_j \rvert < 1/2 \lvert (\sum_{i=2}^m \beta_i \mathbf{x}_i)_j \rvert$.
In the first case we immediately have $\beta_1 \neq 0$, while the second case means $(\sum_{i=1}^m \beta_i \mathbf{x}_i)_j \neq 0$ and hence $\sum_i \beta_i \mathbf{x}_i \neq \mathbf{0}$.
This shows $\sum_i \beta_i \mathbf{x}_i \neq \mathbf{0}$ or $\beta_l \neq 0$ for every $l \neq 2$ and we already know $\beta_2 \neq 0$.
Now for each $k$, decide $\sum_i \beta_i \mathbf{x}_i \neq \mathbf{0}$ or $\beta_k \neq 0$.
We either find $\beta_k \neq 0$ that for all $k$, or there is a $k$ which gives us $\sum_i \beta_i \mathbf{x}_i \neq \mathbf{0}$.
\end{proof}
\end{lemma}

\begin{lemma} \label{l_reso}
Let $A$ be a square $n \times n$ matrix with eigenvalues $\lambda, \mu$ that are apart.
Then $\norm{A - \lambda I} > 0$.
\begin{proof}
Clearly, the determinant is a continuous function $\mathbb{R}^{n \times n} \rightarrow\mathbb{R}$.
Consider $\det$ in the point $(\lambda - \mu) I$ and find a $\delta > 0$ such that $\lvert \det \big((\lambda - \mu) I + E \big) - \det (\lambda - \mu) I \rvert < \lvert \lambda - \mu \rvert^n/2$ for any $E$ with $\norm{E} < \delta$.
Decide $\norm{A - \lambda I} > 0 \lor \norm{A - \lambda I } < \delta$, in the first case we are done.
We will derive a contradiction from the second option $\norm{A - \lambda I } < \delta$.
We can then write $\det (A - \mu I) = \det \big( (A - \lambda I) + (\lambda - \mu)I\big)$ and since the perturbation $(A - \lambda I)$ of $(\lambda - \mu)I$ satisfies $\norm{A - \lambda I } < \delta$, we get $\lvert \det (A - \mu I) - \det (\lambda - \mu) I  \rvert < \lvert \lambda - \mu \rvert^n/2$.
This reduces to $\lvert 0 - \lvert \lambda - \mu\rvert^n  \rvert < \lvert \lambda - \mu \rvert^n/2$ because $\mu$ is an eigenvalue and this is contradictory.
\end{proof}
\end{lemma}

\begin{proposition} \label{p_apart_eigenvectors}
Let $A$ be a square $n \times n$ matrix and let $\lambda_1,...,\lambda_k$ be eigenvalues of $A$ that are pairwise apart.
Then there is a $\varepsilon' > 0$ and real number $f(A, \lambda_1,...,\lambda_k) > 0$ such that for all $0 < \varepsilon < \varepsilon'$, any vector $\beta \in \mathbb{R}^k$ with $\beta_i \neq 0$ for $1 \leq i \leq k$ and any set $\mathbf{x}_i$ with $1 \leq i \leq l$ for some $l \leq n$ of $\varepsilon$-eigenvectors corresponding to $\lambda_1,...,\lambda_k$ the norm of $\sum_{i=1}^l \beta_i \mathbf{x}_i$ is bounded below by $f(A, \lambda_1,...,\lambda_k) \sum_{i = 1}^l \lvert \beta_i \rvert$.
\begin{proof}

Initially, we just consider the case where each eigenvalue has precisely one corresponding $\varepsilon$-eigenvector.
In particular we then have $l = k$.
We prove this with induction on $k$.
The base case is trivial, if $k = 1$ then for any $\varepsilon$, any $\varepsilon$-eigenvector $\mathbf{x}$ and $\beta_1 \in \mathbb{R}$ we have $\norm{\beta_1\mathbf{x}_1} = \lvert \beta_1 \rvert \norm{\mathbf{x}_1} = \lvert \beta_1 \rvert$.
So we define $\varepsilon' = 1$ and $f(A, \lambda_1) = 1$.
\clearpage
\noindent Suppose the theorem is proven for $k < n$, we will show it also holds for $k + 1$.
Define $\Lambda = \{\lambda_1,...,\lambda_{k + 1}\}$, then by the induction hypothesis, for each $1 \leq i \leq k + 1$ there are $\varepsilon'_i > 0$ and real numbers $f(A, \Lambda / \{ \lambda_i \})$ satisfying the requirements of the proposition.
Write $r = \max_{1 \leq i \leq k + 1} \norm{A - \lambda_iI}$, then $r > 0$ by Lemma \ref{l_reso}.
Define the following expressions:
\begin{align*}
    \varepsilon^* &= \min_{1 \leq i \leq k + 1} \varepsilon'_i \\
    f^*(A, \Lambda) &= \min_{1 \leq i \leq k} f(A, \Lambda / \{ \lambda_i \}) > 0 \\
    \hat{\beta} &= (2k + 1) / (2k - 1) \\
    \upgamma &= \min_{i \neq j} \lvert \lambda_i - \lambda_j \rvert > 0 \\
    f(A, \Lambda) &= \frac{\upgamma}{4r\hat{\beta}}f^*(A, \Lambda) \\
    \varepsilon' &= \min \big(\varepsilon^*, \ \frac{\upgamma}{2\hat{\beta}}f^*(A, \Lambda)\big).
\end{align*}
Let $\varepsilon > 0$ satisfy $\varepsilon < \varepsilon'$ and let $\mathbf{x}_i$ be $\varepsilon$-eigenvectors corresponding to $\lambda_i$ respectively for $1 \leq i \leq k + 1$.
Let $\beta$ be a vector in $\mathbb{R}^{k + 1}$ with all entries apart from $0$.
Define $\mathbf{z} = \sum_{i=1}^{k + 1} \beta_i \mathbf{x}_i$, we have to prove $\norm{\mathbf{z}} > f(A, \Lambda) \sum_{i = 1}^{k + 1} \lvert \beta_i \rvert$.
Moreover define $\mathbf{r}_i$ for $1 \leq i \leq k + 1$ by $\mathbf{r}_i = A\mathbf{x}_i - \lambda_i\mathbf{x}_i$.
We then know $\norm{\mathbf{r}_i} < \varepsilon$.
Find a $j$ such that $\lvert \beta_j \rvert < \sum_{i = 1} ^ {k + 1} \lvert \beta_i \rvert / (k + 1/2)$ by deciding $\lvert \beta_i \rvert > \sum_{i = 1} ^ {k + 1} \lvert \beta_i \rvert / (k + 1)$ or $\lvert \beta_i \rvert < \sum_{i = 1} ^ {k + 1} \lvert \beta_i \rvert / (k + 1/2)$ for each $i$ and noting that $\lvert \beta_i \rvert > \sum_{i = 1} ^ {k + 1} \lvert \beta_i \rvert / (k + 1)$ for all $1 \leq i \leq k + 1$ is contradictory.
Thus we get:
\begin{align}
    \sum_{i = 1} ^ {k + 1} \lvert \beta_i \rvert = \Big(\sum_{i = 1, i \neq j} ^ {k + 1} \lvert \beta_i \rvert\Big) + \lvert \beta_j\rvert &< \Big(\sum_{i = 1, i \neq j} ^ {k + 1} \lvert \beta_i \rvert\Big) + \sum_{i = 1} ^ {k + 1} \lvert \beta_i \rvert / (k + 1/2), \notag \\
    \big(1 - 1/(k+1/2)\big)\sum_{i = 1} ^ {k + 1} \lvert \beta_i \rvert &< \sum_{i = 1, i \neq j} ^ {k + 1} \lvert \beta_i \rvert, \notag \\
    \sum_{i = 1} ^ {k + 1} \lvert \beta_i \rvert &< \hat{\beta} \sum_{i = 1, i \neq j} ^ {k + 1} \lvert \beta_i \rvert. \label{eq_beta_hat}
\end{align}
Without loss of generality and for notational convenience assume $j = 1$.
We then have the following:
\begin{align*}
    A\beta_1\mathbf{x}_1 &= A\big(\mathbf{z} - \sum_{i = 2}^{k + 1} \beta_i\mathbf{x}_i \big) = A\mathbf{z} - \sum_{i = 2}^{k + 1} \beta_iA\mathbf{x}_i = A\mathbf{z} - \sum_{i = 2}^{k + 1} \beta_i(\lambda_i\mathbf{x}_i + \mathbf{r}_i) = A\mathbf{z} - \sum_{i = 2}^{k + 1} \beta_i\lambda_i\mathbf{x}_i - \sum_{i = 2}^{k + 1} \beta_i\mathbf{r}_i.
\end{align*}
At the same time, we can write:
\begin{align*}
    A\beta_1\mathbf{x}_1 &= \beta_1 (\lambda_1 \mathbf{x}_1 + \mathbf{r}_1) = \beta_1 \lambda_1 \mathbf{x}_1 + \beta_1\mathbf{r}_1 = \lambda_1 \big(\mathbf{z} - \sum_{i = 2}^{k + 1} \beta_i\mathbf{x}_i \big) + \beta_1\mathbf{r}_1 = \lambda_1 \mathbf{z} - \sum_{i = 2}^{k + 1} \beta_i \lambda_1\mathbf{x}_i + \beta_1\mathbf{r}_1.
\end{align*}
Equating the two expressions and rearranging, we obtain:
\begin{equation} \label{eq_z}
    (A - \lambda_1I)\mathbf{z} - \sum_{i = 1}^{k + 1}\beta_i \mathbf{r}_i = \sum_{i = 2}^{k + 1} \beta_i (\lambda_i - \lambda_1)\mathbf{x}_i.
\end{equation}
Decide $\norm{z} > f(A, \Lambda) \sum_{i = 1} ^ {k + 1} \lvert \beta_i \rvert$ or $\norm{z} < 2f(A, \Lambda)\sum_{i = 1} ^ {k + 1} \lvert \beta_i \rvert $.
In the first case we are done, so we may assume $\norm{z} < 2 f(A, \Lambda)\sum_{i = 1} ^ {k + 1} \lvert \beta_i \rvert$.
We will show this is contradictory, so that in fact we must have  $\norm{z} > f(A, \Lambda) \sum_{i = 1} ^ {k + 1} \lvert \beta_i \rvert$.
First consider the left hand side of (\ref{eq_z}) and start with the term $(A - \lambda_1I)\mathbf{z}$:
\begin{align}
    \norm{(A - \lambda_1I)\mathbf{z}} &\leq \norm{A - \lambda_1I} \norm{z} \leq r \norm{z} < r \cdot 2 f(A, \Lambda)\sum_{i = 1} ^ {k + 1} \lvert \beta_i \rvert = \frac{\upgamma}{2\hat{\beta}} f^*(A, \Lambda) \sum_{i = 1} ^ {k + 1} \lvert \beta_i \rvert. \label{eq_1}
\end{align} 
\clearpage 
\noindent Subsequently we investigate $\sum_{i = 1}^{k + 1}\beta_i \mathbf{r}_i$:
\begin{align}
    \norm{\sum_{i = 1}^{k + 1}\beta_i \mathbf{r}_i} &\leq \sum_{i = 1}^{k + 1} \lvert \beta_i \rvert \norm{\mathbf{r}_i} < \varepsilon' \sum_{i = 1}^{k + 1} \lvert \beta_i \rvert \leq \frac{\upgamma}{2\hat{\beta}} f^*(A, \Lambda) \sum_{i = 1}^{k + 1} \lvert \beta_i \rvert. \label{eq_2}
\end{align}
Combining (\ref{eq_1}) and ($\ref{eq_2}$) we obtain:
\begin{equation*} 
     \norm{(A - \lambda_1I)\mathbf{z} - \sum_{i = 1}^{k + 1}\beta_i \mathbf{r}_i} < \frac{\upgamma}{\hat{\beta}} f^*(A, \Lambda)\sum_{i = 1} ^ {k + 1} \lvert \beta_i \rvert. 
\end{equation*}
We can rewrite this as follows using (\ref{eq_beta_hat}):
\begin{align}
    \norm{(A - \lambda_1I)\mathbf{z} - \sum_{i = 1}^{k + 1}\beta_i \mathbf{r}_i} 
    &< \frac{\upgamma}{\hat{\beta}} f^*(A, \Lambda) \cdot\bigg( \sum_{i = 1} ^ {k + 1} \lvert \beta_i \rvert / \sum_{i = 2} ^ {k + 1} \lvert \beta_i \rvert \bigg) \cdot \sum_{i = 2} ^ {k + 1} \lvert \beta_i \rvert  \notag \\
    &\leq \frac{\upgamma}{\hat{\beta}} f^*(A, \Lambda) \cdot \hat{\beta} \cdot \sum_{i = 2} ^ {k + 1} \lvert \beta_i \rvert  \notag \\
    &= \upgamma f^*(A, \Lambda) \sum_{i = 2} ^ {k + 1} \lvert \beta_i \rvert. \label{eq_3}
\end{align}
Note that we are allowed to divide by $\sum_{i = 2} ^ {k + 1} \lvert \beta_i \rvert$, as this sum is apart from $0$ by our assumption that $\lvert \beta_i \rvert > 0$ for all $i$ and the induction hypothesis $k \geq 2$.
Next we look at the right-hand side of equation (\ref{eq_z}), for which we have a lower bound by our induction hypothesis:
\begin{align}
    \norm{\sum_{i = 2}^{k + 1} \beta_i (\lambda_i - \lambda_1)\mathbf{x}_i} 
    &> \Big( \sum_{i = 2}^{k + 1} \lvert \beta_i (\lambda_i - \lambda_1) \rvert \Big) f(A, \Lambda / \{ \lambda_1 \}) = f(A, \Lambda / \{ \lambda_1 \})\sum_{i = 2}^{k + 1} \lvert \beta_i \rvert \lvert \lambda_i - \lambda_1 \rvert \notag \\
    &\geq f(A, \Lambda / \{ \lambda_1 \}) \sum_{i = 2}^{k + 1} \lvert \beta_i \rvert \upgamma \geq \upgamma f^*(A, \Lambda)) \sum_{i = 2}^{k + 1} \lvert \beta_i \rvert.  \label{eq_4} 
\end{align}
Clearly (\ref{eq_3}) and (\ref{eq_4}) together contradict equation (\ref{eq_z}).

Now we sketch how to prove the case when there are $\varepsilon$-eigenspaces containing more than one $\varepsilon$-eigenvector.
Note first how the induction will then still range over $k$, but now $k$ does not necessarily signify the number of $\varepsilon$-eigenvectors in our linear combination.
It just stands for the number of apart eigenvalues under consideration.

The base case $k = 1$ then looks at all linear combinations of the form $\sum_{i = 1}^l \beta_i \mathbf{x}_i$.
Here, the vectors $\mathbf{x}_i$ are $\varepsilon$-eigenvectors of a yet to be determined $\varepsilon > 0$ corresponding to a single eigenvalue $\lambda$.
In that case we simply take $f(A, \{ \lambda \}) = M_A$, where $M_A$ is the number from Proposition \ref{p_epsilon_kernel}.

For the induction step, consider some $\mathbf{y}_i$ for $1 \leq i \leq m$ that are $\varepsilon$-eigenvectors for a single eigenvalue $\lambda$.
If $\alpha \in \mathbb{R}^m$ and any $0 < \delta < 1$, then we can rewrite the linear combination $\sum_{i = 1}^l \alpha_i \mathbf{y}_i$ as $\alpha_j\sum_{i = 1}^l \alpha_j^{-1} \alpha_i \mathbf{y}_i$ for some $j$ with $\lvert \alpha_j \rvert > \max_i \lvert \alpha_i \rvert / (1 + \delta)$.
Again this uses the fact that $\lvert \alpha_i \rvert > 0$ for all $i$, which in particular means $\lvert \alpha_i \rvert > 0$ for some $i$.
Defining $\mathbf{z} = \sum_{i = 1}^l \alpha_j^{-1} \alpha_i \mathbf{y}_i$, then we have $\norm{ \mathbf{z}} \leq \sum_{i = 1}^l \lvert \alpha_j^{-1} \alpha_i  \rvert \norm{\mathbf{y}_i} = \sum_{i = 1}^l \lvert \alpha_j^{-1} \alpha_i  \rvert \leq \sum_{i = 1}^l (1 + \delta) \leq n(1 + \delta)$.
Moreover similar estimates show that if the vectors $\mathbf{y}_i$ are $\varepsilon$-eigenvectors for the eigenvalue $\lambda$, then $\mathbf{z}$ is $n(1 + \delta) \varepsilon$-eigenvector. 
Thus the whole linear combination $\sum_{i = 1}^l \alpha_i \mathbf{y}_i$ of $\varepsilon$-eigenvectors can be seen as $\alpha_j \mathbf{z}$, where $\alpha_j \neq 0$ and $\mathbf{z}$ is a $n(1 + \delta) \varepsilon$-eigenvector. 
The proof for the geometric multiplicities all equal to 1 case then still works if we group together all $\varepsilon$-eigenvectors of the same eigenvalue and adjust the definition of $f(A, \Lambda)$ to include an extra factor $1 / \big( n(1 + \delta)\big)$.
The coefficients $\beta$ will then be the above defined $\alpha_j$, one for each eigenvalue.
\end{proof}
\end{proposition}

The proof sketch of Proposition \ref{p_apart_eigenvectors} in the case of geometric multiplicities greater than $1$ is a bit cumbersome. 
A more direct proof, involving an equality as (\ref{eq_z}), will likely also work.
This will in some sense be simpler than the construction we used.
Perhaps it will even yield tighter bounds for $\sum_{i=1}^l \beta_i \mathbf{x}_i$.
However, this will come at the cost of a heavy notational burden. 
We would  have to keep track of both the index of the corresponding eigenvalue and the index within the respective $\varepsilon$-eigenspace of every $\varepsilon$-eigenvector.

\begin{proposition}\label{p_eigen}
Let $A$ be a square $n \times n$ matrix and let $\lambda_1,...,\lambda_k$ be eigenvalues of $A$ that are pairwise apart.
Then there is an $\varepsilon' > 0$ such that for all $0 < \varepsilon \leq \varepsilon'$, any set of $\varepsilon$-eigenvectors corresponding to $\lambda_1, ...,\lambda_k$ are independent.
\begin{proof}
We use induction on $k$.
If $k = 1$, then we can reason exactly as in the previous proof.
Assume the result is true for $k < n$, we prove it for $k + 1$.
We have to show $\sum_{i = 1}^{k + 1} \beta_i \mathbf{x}_i \neq \mathbf{0}$ for all $\beta \in \mathbb{R}^k$ with $\beta \neq \mathbf{0}$.
Using Lemma \ref{l_beta}, decide $\sum_{i = 1}^{k + 1} \beta_i \mathbf{x}_i \neq \mathbf{0}$ or $\beta_i \neq 0$ for all $i$.
In the first case we are done.
That $\sum_{i = 1}^{k + 1} \beta_i \mathbf{x}_i \neq \mathbf{0}$ also holds in the second case is precisely the content of Proposition \ref{p_apart_eigenvectors}.
\end{proof}
\end{proposition}

We conclude with our main result Theorem \ref{t_eigen}.
At a quick glance, the difference between Proposition \ref{p_eigen} and this theorem is not directly clear.
The crucial difference is that the $\varepsilon'$ in Theorem \ref{t_eigen} does not depend on the eigenvalues.
Thus, given a matrix $A$, Theorem \ref{t_eigen} gives a set of eigenvalues $\Lambda$ and a set of vectors $X$ such these both are `approximations up to $\varepsilon$', stated informally.
The set $X$ consists of $\varepsilon$-eigenvectors and for every eigenvalue $\mu$ of $A$ there is a $\lambda \in \Lambda$ such that $\lvert \mu - \lambda \rvert < \varepsilon$.

\begin{theorem} \label{t_eigen}
Let $A$ be a square $n \times n$ matrix and $\varepsilon' > 0$.
Denote with $\lambda'_1,...,\lambda'_n$ the eigenvalues of $A$. 
Then there is a $0 < \varepsilon \leq \varepsilon'$ such that there are:
\begin{enumerate}
    \item real numbers $\lambda_1,...,\lambda_k$ with $k \leq n$, all eigenvalues of $A$,
    \item for each eigenvalue $\lambda'_i$ of $A$ there is a $j \leq k$ such that $\lvert \lambda'_i - \lambda_j \rvert < \varepsilon$,
    \item for each eigenvalue $\lambda_1,...,\lambda_k$ there exist $\varepsilon$-eigenvectors,
    \item the collection of these $\varepsilon$-eigenvectors is independent.
\end{enumerate}
\begin{proof}
Start with $\varepsilon_0 = \frac{1}{2}\varepsilon'$ and for each pair $\lambda_i', \lambda_j'$ with $i \neq j$, decide $\lvert \lambda'_i - \lambda_j' \rvert > 0$ or $\lvert \lambda'_i - \lambda_j' \rvert < \varepsilon_0$.
This way, we can separate the numbers $\lambda'_1,...,\lambda'_n$ into a set $\Lambda_0$ of pairwise apart numbers, and a set $X_0$ such that for every $\mu \in X$ there is a $\lambda \in \Lambda_0$ with $\lvert \mu - \lambda \rvert < \varepsilon_0$.
For notational convenience, relabel the elements of $\Lambda_0$ such that $\Lambda = \{\lambda'_1,...,\lambda'_k \}$ with $k \leq n$.
By Proposition \ref{p_eigen}, we can find an $\varepsilon[\Lambda_0] > 0$ such that the $\delta$-eigenvectors corresponding to the eigenvalues in $\Lambda_0$ are independent for all $0 < \delta < \varepsilon[\Lambda_0]$.
Next, define $\varepsilon_1 = \min(\varepsilon_0, \varepsilon[\Lambda_0])$ and for each $\mu \in X_0$ and each $\mu 
\in X$ and $1 \leq i \leq k$, decide $\lvert \mu - \lambda'_i \rvert > 0 \lor \lvert \mu - \lambda'_i \rvert < \varepsilon_1$.
If for every $\mu$ there is a $j$ with $\lvert \mu - \lambda'_j \rvert < \varepsilon_1$, then we are done.
We define $\varepsilon = \varepsilon_1$, since then $\varepsilon \leq \varepsilon[\Lambda_0]$ and the $\varepsilon$-eigenvectors corresponding to $\Lambda_0$ are all independent.
Suppose there are $\mu$ for which this is not the case, so $\mu$ that are apart from every element of $\Lambda_0$.
Then define $\Lambda_1$ as the union of $\Lambda_0$ and these values $\mu$ and let $X_1$ be those $\mu$ of $X_0$ that remain.
Now $\Lambda_1$ again consists of pairwise apart eigenvalues of $A$.
We again find $\varepsilon[\Lambda_1] > 0$ this way such that the $\delta$-eigenvectors corresponding to the eigenvalues in $\Lambda_1$ are independent for all $0 < \delta < \varepsilon[\Lambda_1]$.
Likewise, we define $\varepsilon_2 = \min(\varepsilon_1, \varepsilon[\Lambda_1])$ and again see for all $\mu \in X_1$ whether they are closer than $\varepsilon_2$ to some element of $\Lambda_1$, or apart from all elements of $\Lambda_1$.
If all $\mu \in X_1$ fall in the first category we define $\varepsilon = \varepsilon_2$ and we are done, otherwise we construct $\Lambda_2$, $X_2$, $\varepsilon_3 < \varepsilon_2$, $\varepsilon_3$-eigenvectors corresponding to $\Lambda_2$, etc. 
This construction must end at some point.
For any $k$, the set $\Lambda_{k + 1}$ is strictly greater than $\Lambda_{k}$ and we know that any such set can contain at most $n$ elements.
Thus we must find an $\varepsilon = \varepsilon_k$ with $k \leq n$ satisfying our requirements.
\end{proof}
\end{theorem}
\clearpage

\bibliographystyle{alpha}
\bibliography{cla}

@article{CAoEPiCuNU,
    author = "Pavel Osinenko and Grigory Devadze and Stefan Streif",
    title = "Constructive Analysis of Eigenvalue Problems in Control under Numerical Uncertainty",
    journal = "Control Theory and Applications",
    year = 2020,
    volume = "18",
    pages = "2177-–2185"
}

@article{CtDoLS,
    author = "Martin Ziegler and Vasco Brattka",
    title = "Computing the Dimension of Linear Subspaces",
    journal = "SOFSEM 2000: Theory and Practice of Informatics",
    year = 2000,
    pages = "450-–458"
}

@article{DoCMaaRoI,
    author  = "Andre Scedrov",
    title   = "Diagonalization of Continuous Matrices as a Representation of Intuitionistic Reals",
    year    = 1986,
    journal = "Annals of Pure and Applied Logic",
    volume  = "30",
    pages   = "201--206"
}

@article{EAoDMoMI,
    author  = "J. H. Wilkinson",
    title   = "Error Analysis of Direct Methods of Matrix Inversion" ,
    year    = 1961,
    journal = "Journal of the ACM" ,
    volume  = "8",
    number  = "3",
    pages   = "281--330"
}

@article{UuIA,
    author  = "Arend Heyting",
    title   = "Untersuchungen {\"u}ber intuitionistische algebra",
    year    = 1941,
    journal = "Verhandelingen der Nederlandsche Akademie van Wetenschappen",
    volume  = "18",
    number  = "2",
    pages   = "1--36"
}

@article{SoSOiCA,
  author    = "Douglas Bridges and Hajime Ishihara",
  title     = "Spectra of selfadjoint operators in constructive analysis",  year      = 1996,
  journal   = "Indagationes Mathematicae",
  volume    = "7",
  number    = "1",
  pages     = "11--35",

}

@book{ACiCA,
    author  = "Ray Mines and Fred Richman and Wim Ruitenburg",
    title   = "A Course in Constructive Algebra",
    year      = 1988,
    publisher = "Springer",
    address   = "New York"
}

@book{CA,
    author  = "Errett Bishop and Douglas Bridges",
    title   = "Constructive Analysis",
    year      = 1985,
    publisher = "Springer",
    address   = "New York"
}

@book{CiM2,
    author  = "A.S. Troelstra and D. van Dalen",
    title   = "Constructivism in Mathematics: An Introduction. Volume 2",
    year      = 1988,
    publisher = "Elsevier",
    address   = "Amsterdam"
}

@book{ToCA,
    author  = "Douglas S. Bridges and Luminita Simona Vita",
    title   = "Techniques of Constructive Analysis",
    year      = 2006,
    publisher = "Springer",
    address   = "New York"
}

@book{VoCM,
    author  = "Douglas S. Bridges and Fred Richman",
    title   = "Varieties of Constructive Mathematics",
    year      = 1987,
    publisher = "Cambridge University Press",
    address   = "Cambridge"
}

@incollection{FoITitEoNaMC,
  author       = {Vasco Brattka and Martin Ziegler},
  title        = {Computability of Linear Equations},
  booktitle    = {Foundations of Information Technology in the Era of Network and Mobile Computing},
  publisher    = {Springer},
  year         = {2002},
  pages        = {95--106},
  address      = {New York}
}

@incollection{MUCaP,
  author       = {Josef Berger and Gregor Svindland},
  title        = {Constructive Proofs of Negated Statements},
  booktitle    = {Mathesis Universalis, Computability and Proof },
  publisher    = {Springer},
  year         = {2019},
  pages        = {47--53},
  address      = {New York}
}

@phdthesis{CRM,
  author       = {Hannes Diener},
  title        = {Constructive Reverse Mathematics},
  school       = {Universität Siegen},
  year         = {2018},
  type         = {Habilitationsschrift},
  address      = {Siegen, Germany}
}

@phdthesis{HSiI,
  title        = "Hilbert Spaces in Intuitionism",
  author       = "Ashwini Kumar",
  year         = 1966,
  school       = "Universiteit van Amsterdam",
  type         = {PhD thesis}
}

@phdthesis{InA,
  title        = "Intuitionistic Algebra",
  author       = "Wim Ruitenburg",
  year         = 1982,
  school       = "Universiteit Utrecht",
  type         = {PhD thesis}
}

\end{document}